\documentclass[mnsc,nonblindrev]{informs4}

\usepackage{booktabs}
\usepackage{bm}
\usepackage{dsfont}
\usepackage{extarrows}

\usepackage[table,xcdraw]{xcolor}

\newcommand{\Px}{ P }
\newcommand{\Qx}{ \mathbb{Q} }

\newcommand{\Ex}{ E }

\numberwithin{equation}{section}

\def\esssup_#1{\underset{#1}{\mathrm{ess\,sup\, }}}
\def\essinf_#1{\underset{#1}{\mathrm{ess\,inf\, }}}
\def\argmax_#1{\underset{#1}{\mathrm{arg\,max\, }}}
\def\argmin_#1{\underset{#1}{\mathrm{arg\,min\, }}}

\newcommand{\Fx}{\mathbb{F} }

\newcommand{\lc}{\langle}
\newcommand{\rc}{\rangle}

\newcommand{\F}{\mathcal{F}}
\newcommand{\R}{\mathbb{R}}

\newcommand{\I}{\mathds{1}}
\newcommand{\LW}{\overline{W}}
\newcommand{\Ac}{{\bf(A$_{C}$)}}
\newcommand{\Ad}{{\bf(A$_{D}$)}}
\newcommand{\Af}{{\bf(A$_{F}$)}}

\newcommand{\ASP}{$P(d\omega)\otimes dt$-a.s.}
\newcommand{\ASQ}{$Q^M(d\omega)\otimes dt$-a.s.}

\definecolor{linkcolor}{rgb}{0,0,0.502}
\definecolor{urlcolor}{rgb}{1,0,0}

\usepackage{natbib}
\usepackage{subfigure}
 \NatBibNumeric
 \def\bibfont{\small}%
\definecolor{DRed}{rgb}{0.5,  0.00,  0.00}

\definecolor{Blue}{rgb}{0.00,  0.00,  1.00}

\definecolor{Green}{rgb}{0.0,  0.4,  0.0}

\newcolumntype{I}{!{\vrule width 1.2pt}}
\newlength\savedwidth

\newlength\savewidth

\usepackage[colorlinks=true, breaklinks=true, bookmarks=true, urlcolor=blue,
     citecolor=blue, linkcolor=blue, bookmarksopen=false, draft=false]{hyperref}

\def\EMAIL#1{\href{mailto:#1}{#1}}
\def\URL#1{\href{#1}{#1}}

\newtheorem{theorem}{Theorem}[section]
\newtheorem{definition}{Definition}[section]

\newtheorem{proposition}[theorem]{Proposition}
\newtheorem{remark}[theorem]{Remark}
\newtheorem{lemma}[theorem]{Lemma}

\newtheorem{example}{Example}[section]
\usepackage{amssymb}

\begin{document}
\RUNAUTHOR{L. Bo, A. Capponi and C. Zhou}
\RUNTITLE{Power Forward Performance Process}

\TITLE{Power Forward Performance in Semimartingale Markets with Stochastic Integrated Factors}
\ARTICLEAUTHORS{
\AUTHOR{Lijun Bo}
\AFF{School of Mathematics and Statistics,  Xidian University, Xi'an 710126, P.R. China\\ \EMAIL{lijunbo@xidian.edu.cn}}
\AUTHOR{Agostino Capponi}
\AFF{Department of Industrial Engineering and Operations Research, Columbia University, New York, 10027, NY, USA\\ \EMAIL{ac3827@columbia.edu}}
\AUTHOR{Chao Zhou}
\AFF{Department of Mathematics and Risk Management Institute, National University of Singapore, Singapore\\ \EMAIL{matzc@nus.edu.sg}  \URL{}}
}

\ABSTRACT{We study the forward investment performance process (FIPP) in an incomplete semimartingale market model with closed and convex portfolio constraints, when the investor's risk preferences are of the power form. We provide necessary and sufficient conditions for the existence of such FIPP. In a semimartingale factor model, we show that the FIPP can be recovered as a triplet of processes {which admit an integral representation with respect to semimartingales}. Using an integrated stochastic factor model, {we relate the factor representation of the triplet of processes} to the smooth solution of an ill-posed partial integro-differential Hamilton-Jacobi-Bellman (HJB) equation. We develop explicit constructions for the class of time-monotone FIPPs, generalizing existing results from Brownian to semimartingale market models.

\vspace{0.1cm}
}

\maketitle

\section{Introduction}

The classical approach to expected utility maximization, pioneered by \cite{Merton}, is to fix a planning horizon, and specify a utility function to measure the investment performance at the end of the horizon. Despite Merton's portfolio problems admitting analytically tractable solutions with direct economic interpretations, they fail to capture several behavioral features of investor's decision making. Importantly, the chosen utility function is static, and the investor's risk preferences are specified once for all future times.

To overcome these limitations, a novel approach to portfolio selection was introduced by \cite{MusiZari08}. This approach does not require the investor to fix her risk preferences beforehand, but rather gives her the flexibility to dynamically adapt them. Concretely, as opposed to fixing a market model and a terminal utility, the investor starts with an initial investment performance measure, and then updates it over time as the factors driving the market dynamics evolve. The evolution of this forward investment performance process, or FIPP, is driven by a forward-in-time version of the dynamic programming principle, designed to guarantee time consistency. In financial terms, this means that the investor updates her risk preferences so that her past investment choices, viewed from the perspective of her current risk preferences, are still optimal. Throughout the paper, we will also refer to this process as a forward performance process, or simply as a forward process.

{The forward utility framework allows capturing risk preferences which evolve over time and which are sensitive to changes in investor's wealth. The study of \cite{StrubZhou2020} shows that the curvature of the Arrow--Pratt risk-tolerance measure determines how risk preferences change with passage of time. Their theoretical findings imply that investors become more risk-tolerant as they grow older. 
\cite{StrubZhou2020} additionally show that risk preferences are constant with respect to time if and only if the initial utility function of a forward utility pair exhibits constant relative risk aversion. It thus follows that the class of power forward preferences considered in this paper possesses a risk-tolerance measure which is invariant over time. Nevertheless, these preferences are still allowed to depend on the investor's wealth.}

The evolution of the forward performance process can in general be described by a stochastic partial differential equation (SPDE) which, due to degeneracy and nonlinearity, presents several technical challenges. In particular, the volatility of the forward performance process  depends on the process itself and its gradient. Hence, existing results on existence and uniqueness of solutions to fully non-linear SPDEs are not applicable (see, for instance, \cite{MusiZari10}). \cite{MusiZari10} and \cite{MusiZari10a} study the case of zero volatility, and find that the solution to the corresponding SPDE can be characterized by a class of so-called time-monotone processes. A time-monotone process is a composition of deterministic and stochastic inputs, where the deterministic input is a space-time function satisfying an ill-posed PDE, and the stochastic input is a finite variation process. Another widely studied setting is that of homothetic forward performance processes, in which the dependence on the investor's wealth is either {exponential or of the} power form. \cite{Zitkovic} provides a dual characterization of the exponential forward performance processes in a general semimartingale market model, and develops an explicit parametrization for markets driven by It\^{o}-processes. \cite{choulli17} study power forward performance processes in a locally bounded semimartingale market, but without portfolio constraints. A few studies have considered forward performance processes in the factor form (see \cite{MusiZari10}, \cite{NadTeh17}, \cite{Shkolnikov} and \cite{AvanesyanShkolnikovSircar18}), in which the SPDE simplifies to a so-called forward HJB equation.  In a complete Black-Scholes (BS) market with stochastic factors, \cite{NadTeh17} characterize both the power forward performance process and the optimal strategy in terms of a forward HJB equation defined on a semi-infinite time interval. For the case of incomplete BS markets, the portfolio selection problem under the forward performance criterion has been investigated by \cite{Shkolnikov}, who develop an expansion formula for the underlying ill-posed HJB equation. \cite{AvanesyanShkolnikovSircar18} analyze forward performance processes of the factor form, in which the randomness only enters through a stochastic factor process. They solve the associated ill-posed PDE, and generalize the Widder's theorem of \cite{NadTeh17} to establish a local forward investment performance process (FIPP) of separable power factor form.

Our work is also related to recent studies on forward investment, consumption and entropic risk measures by \cite{LiangZari17}, \cite{chongetal2018} and \cite{chongliang2018}. All these studies focus on Brownian driven markets. \cite{LiangZari17} and \cite{chongetal2018} apply the so-called discount BSDE method, previously used in the analysis of ergodic BSDEs (see, e.g., \cite{Richou2009}, \cite{cohenhu2013} and \cite{LiZhao19}), to establish existence and uniqueness of Markovian solutions for their infinite horizon BSDEs. They introduce an arbitrary positive parameter in the  BSDE to make the driver strictly monotone in the first solution component, and then apply Theorem 3.3 in \cite{BriandConfortola} to obtain existence and uniqueness of their BSDE solution.

In this paper, we study the power forward performance process in a semimartingale market with portfolio constraints. We provide necessary and sufficient conditions for the existence of such a process, exploiting the multiplicative decomposition of a strictly positive special semimartingale. {Because of portfolio constraints, the optimal trading strategy cannot be characterized via a first-order condition as in \cite{choulli17}. Rather, the optimization problem yielding the optimal trading strategy is a constrained extremum problem. Using results from convex analysis, we transform the constrained problem to an unconstrained one, and explicitly characterize the recession function and cone of the objective function of the transformed problem when the return process is of the L\'evy type.} In a factor market model $(R,Y)$, where both the return process $R$ and the factor process $Y$ are semimartingales, we use the Jacod's decomposition to relate {the construction of a power FIPP to a triplet of processes, whose first component admits an integral representation w.r.t. the semimartingale $Y$.} We then extend our analysis to  time-monotone processes in the semimartingale factor form, by allowing the forward performance criterion to depend on the integral of the semimartingale factor process, and hence on its sample path. We establish closed-form representations of forward processes in this extended setup by characterizing the input function as the classical solution of a forward HJB equation. We construct such a solution using the factor representation of {the above discussed triplet of processes.}

The paper proceeds as follows. In Section \ref{sec:themodel}, we describe the semimartingale market model. In Section \ref{sec:forwardutility}, we provide necessary and sufficient conditions for the existence of power FIPPs in this setting. In Section \ref{sec:BSDE}, we study the structure of power FIPPs in a market model enhanced with an additional semimartingale factor process, and relate them to {a triplet of processes which admit an integral representation w.r.t. the semimartingale factor.} In Section~\ref{sec:FHJB}, we relate the factor representation of {the triplet} to the smooth solution of a  forward HJB equation. Some technical proofs of supporting results are delegated to the Appendix. We collect  mathematical symbolism in the following table, and refer to it whenever needed in the main body.

\noindent {\bf List of symbols}

\noindent {\footnotesize\begin{tabular}{ll}
		\hline
		${\cal B}(\R^n)$ & Space of Borel sets on $\R^n$\\
		$E^Q[~\cdot~]$ & Expectation under a probability measure $Q$\\
		$[M,M]^Q$ & Quadratic variation process $([M^{i},M^{j}]^Q)_{i,j=1,\ldots,n}$ under $Q$\\
		$\lc M,M\rc^Q$ & Predictable compensator of $[M,M]^Q$ under $Q$\\
		${\cal O}$ (${\cal P}$) & $\sigma$-algebra on $\Omega\times[0,\infty)$, generated by c\`adl\`ag (continuous) adapted processes\\
		$(\tilde{\Omega},\tilde{\cal O},\tilde{\cal P})$ & ($\Omega\times[0,\infty)\times\R^n$, ${\cal O}\otimes{\cal B}(\R^n)$, ${\cal P}\otimes{\cal B}(\R^n)$)\\
		${\cal A}_{loc}^{Q,+}$ & Set of locally $Q$-integrable increasing processes \\
		${\cal M}_{loc}^Q$ (${\cal M}_{loc}^{Q,c}$) & Set of (continuous) $Q$-local martingales vanishing at time $0$\\
		${\cal M}_{loc}^{Q,+}$ & Set of scalar local martingales $M\in{\cal M}_{loc}^Q$ satisfying $1+\Delta M>0$\\
		${\cal V}^Q$ & Set of adapted processes with finite variation under $Q$\\
		${\cal H}^{Q,2}$ (${\cal H}_{loc}^{Q,2}$) & Set of (locally) square-integrable martingales under $Q$\\
		$L^{Q,2}(M)$ ($L_{loc}^{Q,2}(M)$) & Set of predictable processes $H$ s.t. $H^{\top}\cdot\lc M^c,M^c\rc^{Q}H$ is  (locally)\\
~~~~& $Q$-integrable if $M\in{\cal H}_{loc}^{Q,2}$\\
	
		\hline
\end{tabular}}

\section{The Model}\label{sec:themodel}

This section presents the modeling framework. In  Section~\ref{sec:marketmodel}, we describe the market return process. In Section~\ref{sec:FIPP}, we define the set of admissible trading strategies, and the forward investment performance criterion. In  Section~\ref{sec:constraints}, we introduce the set of portfolio constraints. Throughout the paper, let $(\Omega,\F,\Px)$ be the original probability space where the filtration $\Fx=(\F_t)_{t\geq0}$ satisfies the usual conditions and $\F_0$ is $\Px$-trivial. All processes considered are assumed to be $\Fx$-adapted.

\subsection{Semimartingale market model}\label{sec:marketmodel}
We consider an $n$-dimensional return process described by a c\`{a}dl\`ag semimartingale $R{=}(R_t)_{t\geq0}$. Let $(B,C,\nu)$ be the predictable characteristics of $R$ relative to a truncation function $h: \R^n\to\R^n$. Then, the canonical representation of $R$ is given by (cf.\ Theorem II.2.34 in \cite{JacodShiryaev}), $P$-a.s.
\begin{align}\label{eq:Rdecom}
R = B+ R^c + h(u)*(\mu-\nu) + (u-h(u))*\mu,
\end{align}
where $B\in{\cal P}\cap{\cal V}^P$ is the predictable finite variation (FV) part of $R$, $R^c \in{\cal M}_{loc}^{P,c}$ is the continuous local martingale part of $R$, and $\mu$ is the random measure for the jumps of $R$ with predictable compensator $\nu$. We assume that the predictable characteristics $(B,C,\nu)$ admit the form:
\begin{align}\label{eq:bcaR0}
B=\int_0^{\cdot}b_tdt,\quad C=\int_0^{\cdot}c_tdt,\quad \nu(dt,du)=F_t(du)dt,
\end{align}
where $F$ is a predictable kernel satisfying $F_t(\{0\})=0$, $\int(|u|^2\wedge1)F_t(du)<\infty$, and $b,c$ are predictable processes. It follows from \eqref{eq:bcaR0} that $R$ is quasi-left-continuous (QLC), which implies that all jumps of $R$ are totally inaccessible stopping times. 
The price process of the $n$ risky assets is given by the Dol\'eans-Dade exponential ${\cal E}(R)$. In addition to the risky assets, the investor has also access to a money market account which accrues zero interest rate. An $\R^n$-valued trading strategy $\pi$ is a predictable $R$-integrable process, where the $i$-th entry denotes the fraction of wealth invested in the $i$-th risky asset. Using the self-financing condition, the wealth process  satisfies $X^{\pi,x} = x + X_{-}^{\pi,x}\pi\cdot R$, where $x>0$ is the investor's initial wealth and $\pi\cdot R:=\int_0^{\cdot}\pi_tdR_t$ is the stochastic integral w.r.t. $R$. Hence, we can write $X^{\pi,x}=x{\cal E}(\pi\cdot R)$.

\subsection{Admissible trading strategies and FIPP}\label{sec:FIPP}

We define the class of admissible trading strategies and the forward investment performance process (FIPP) {under an arbitrary probability measure equivalent to $P$ with density process ${\cal E}(M)$}, where $M\in{\cal M}_{loc}^{P,+}$. More precisely, for any $M\in{\cal M}_{loc}^{P,+}$, we define the following class of probability measures:
\begin{align}\label{eq:setQ}
{\cal Q}_M := \left\{Q\sim P;~\frac{dQ}{dP}\Big|_{\F_t}={\cal E}(M)_t,~t\geq0\right\}.
\end{align}
Obviously, $P\in{\cal Q}_M$ with $M\equiv1$. Let $Q\in{\cal Q}_M$, ${\cal C}_t(\omega):[0,\infty)\times\Omega\to{\cal B}(\R^n)$ be a predictable set-valued process and $U_t^Q(x,\omega):[0,\infty)\times(0,\infty)\times\Omega\to\R$ be a progressively measurable random field. We remark that any FIPP is necessarily of this form, i.e., a stochastic process of utilities that, at each time $t$, depends on the investor's current wealth amount $x$.
\begin{definition}[Set of Admissible Trading Strategies]\label{def:control-set}
An $\R^n$-valued predictable process $\pi$ is said to be ${\cal C}$-constrained under $Q$ if $\pi_t(\omega)\in{\cal C}_t(\omega)$ for all $(t,\omega)\in[0,\infty)\times\Omega$. The class $\Gamma_{\cal C}^Q$ of admissible investment strategies associated with $U^Q$ includes all ${\cal C}$-constrained, predictable {and} $R$-integrable processes $\pi$ that satisfy $\pi^{\top}\Delta R\geq-1$, $Q$-a.s., and s.t. for any $x>0$, $U_{t}^Q(x{\cal E}_{t}(\pi\cdot R))^{-}$, $t\geq0$, is of class {\rm(DL)} under $Q$.
\end{definition}

We next provide the definition of a FIPP with constraint set ${\cal C}$ under a probability measure $Q\in{\cal Q}_M$, {where} $M\in{\cal M}_{loc}^{P,+}$ (see \cite{MusiZari10} for the definition in the absence of portfolio constraints):
\begin{definition}[FIPP]\label{def:FIPP}
{The random field $U^Q$ is a $Q$-FIPP with optimal trading strategy $\pi^*\in\Gamma_{\cal C}^Q$ if it satisfies the following self-generating properties:
\begin{itemize}
  \item[{\rm(i)}] $Q$-a.s., for any $t\geq0$, $x\to U_t^Q(x)$ is concave and increasing;
  \item[{\rm(ii)}] For any $x>0$, and $\pi\in\Gamma_{\cal C}^Q$, the process $U^Q(X^{\pi,x}):=(U_t^Q(x{\cal E}_{t}(\pi\cdot R)))_{t\geq0}$ is a $Q$-supermartingale;
  \item[{\rm(iii)}] For any $x>0$, $U^Q(X^{\pi^*,x})$ is a $Q$-martingale.
\end{itemize}}
\end{definition}

\subsection{Budget constraint}\label{sec:constraints}

For any $\pi\in\Gamma_{\cal C}^Q$ {where} $Q\in{\cal Q}_M$ {and} $M\in{\cal M}_{loc}^{P,+}$, the condition $\pi^{\top}\Delta R\geq-1$ in Definition \ref{def:control-set} is related to the budget constraint of an admissible trading strategy, as specified by~\cite{Kara-Kard}. Let $(B^{M},C^{M},\nu^{M})$ be the predictable characteristics of $R$ under $Q$. It follows from Theorem~III.3.24 in \cite{JacodShiryaev} that
\begin{align}\label{eq:bcaR}
B^{M}=B+\int_0^{\cdot}c_s\alpha_sds+h(u)(\xi(u)-1)*\nu,\quad C^{M}=C,\quad \nu^{M}=\xi*\nu,
\end{align}
where $\alpha=(\alpha^1,\ldots,\alpha^n)^{\top}$ is a predictable process and $\xi$ is a $\tilde{\cal P}$-measurable nonnegative process, which are determined by
\begin{align}\label{eq:alphaphieqn}
\xi=1+M_{\mu}^P(\Delta M|\tilde{{\cal P}}),\quad \lc M^c,R^c\rc^P=\int_0^{\cdot}c_s\alpha_sds.
\end{align}
Here, denote by $M_{\mu}^{P}(\cdot)$ the positive measure such that $M_{\mu}^{P}(W) = E^{P}[(W*\mu)_{\infty}]$ for all measurable nonnegative functions $W$, and $M_{\mu}^P(W|\tilde{\cal P})$ the $M_{\mu}^P$-a.s. unique $\tilde{\cal P}$-measurable function $W'$ s.t. $M_{\mu}^P(W U) = M_{\mu}^P(W'U)$ for all nonnegative $\tilde{\cal P}$-measurable functions $U$ (see \cite{JacodShiryaev}, page 170).
\begin{remark}\label{rem:changexialpha}
Theorem~III.3.24 in \cite{JacodShiryaev} gives
\begin{align}\label{eq:alphaphieqn0}
\xi=\frac{1}{{\cal E}(M)_{-}}M_{\mu}^P({\cal E}(M)|\tilde{{\cal P}}),\quad \lc {\cal E}(M)^c,R^{c}\rc^P=\int_0^{\cdot}{\cal E}(M)_{s}c_s\alpha_sds.
\end{align}
It follows from \eqref{eq:alphaphieqn0} that ${\cal E}(M)_{-}(\xi-1)=M_{\mu}^P(\Delta{\cal E}(M)|\tilde{{\cal P}})={\cal E}(M)_{-}M_{\mu}^P(\Delta M|\tilde{{\cal P}})$. Because ${\cal E}(M)>0$, it then follows that $\xi=1+M_{\mu}^P(\Delta M|\tilde{{\cal P}})$. Moreover, using the relation $\lc{\cal E}(M)^c,R^c\rc^P={\cal E}(M)_{-}\cdot\lc M^c,R^c\rc^P$, it follows from \eqref{eq:alphaphieqn0} that $\lc M^c,R^c\rc^P=\int_0^{\cdot}c_s\alpha_sds$. We then obtain the identities in \eqref{eq:alphaphieqn}.
\end{remark}

In view of \eqref{eq:bcaR}, the differential characteristics of $R$ under $Q$ is given by
\begin{align}\label{eq:diffbcaR}
b^{M}=b+c\alpha+h(u)(\xi(u)-1)* F,~~c^{M}=c,~~ F^M(du)=\xi(u)F(du).
\end{align}
Then, using \eqref{eq:Rdecom} and \eqref{eq:bcaR}, it follows that the continuous local martingale part of $R$ under $Q$ is
\begin{align}\label{eq:RMci}
R^{M,c} = R^{c} - \int_0^{\cdot}c_s\alpha_sds.
\end{align}
The budget constraint of an admissible strategy under $Q$ can be formally specified via the set-valued mapping
\begin{align}\label{eq:closeU}
{\cal C}_0^{M} := \{\pi\in\R^n;\ {F}^{M}(\{u\in\R_0^n;\ \pi^{\top}u+1<0\})=0\},
\end{align}
where $\R_0:=\R\setminus\{0\}$. Note that ${\cal C}_0^{M}$ depends on $(t,\omega)$, because ${F}^{M}$ does. To lighten notation, we avoid to explicitly specify the dependence on {$(t,\omega)$}. The predictability of such a process follows directly from Lemma 9.4 in~\cite{Kara-Kard}.  Moreover, ${\cal C}_0^{M}$ is closed, convex and contains the origin. Thus, $\pi_t(\omega)\in{\cal C}_t(\omega)\cap{\cal C}_{0,t}^M(\omega)$, for any $\pi\in\Gamma_{\cal C}^Q$. Let ${\cal C}_0$ be the set-valued mapping \eqref{eq:closeU} with $F^M$ replaced by $F$, i.e., ${\cal C}_0$ is the budget constraint of an admissible strategy under $P$. Then, it can be easily verified that ${\cal C}_0={\cal C}_0^{M}$, $P(d\omega)\otimes dt$-a.s.
We also introduce the predictable set-valued process of null-investment given by
\begin{align}\label{eq:NQ}
{\cal N}^Q:=\left\{\pi\in\R^n;\ c^{M}{\pi}=0,\ F^{M}(\{u\in\R_0^n;\ \pi^{\top}u\neq0\})=0,\ \pi^{\top}b^{M}=0\right\}.
\end{align}
Under $Q\in{\cal Q}_M$, the wealth process remains the same if one invests according to the strategy specified by a vector in ${\cal N}^Q$. For any $(t,\omega)\in[0,\infty)\times\Omega$,  ${\cal N}_t^Q(\omega)$ is a linear subspace of $\R^n$ and hence an affine set which contains the origin. This implies that it is closed (it is also relatively open).

{We next list the technical assumptions under which the results in this paper will be derived: 
\begin{itemize}
  \item[{\Ac}] For $(t,\omega)\in[0,\infty)\times\Omega$, the constraint set ${\cal C}_t(\omega)$ in Definition \ref{def:control-set} is compact, convex and contains the origin.
\end{itemize}
Under Assumption {\Ac}, the recession cone of the constraint set ${\cal C}_t(\omega)$ is given by $0^+{\cal C}_t(\omega):=\cap_{\lambda>0}\lambda{\cal C}_t(\omega)$ (see, e.g., Corollary 8.3.2 in~\cite{Rockafellar}). 
\begin{itemize}
  \item[{\Af}] For $p\in(0,1)$, under the probability measure $Q\in{\cal Q}_M$, the predictable kernel of $R$ given in \eqref{eq:diffbcaR} satisfies
  \begin{align}\label{eq:xpintegrableFR}
\int_{|u|>1}|u|^p F^{M}(du)<+\infty,\quad Q(d\omega)\otimes dt\mbox{-a.s.}
\end{align}
\end{itemize}

We also introduce a scalar exponentially special $P$-semimartingale $D=(D_t)_{t\geq0}$ which will be used in Section~\ref{sec:forwardutility} to define the power FIPP. Let $(B^D,C^D,\nu^D(dt,dv))=(\int_0^{\cdot}b_s^Dds,\int_0^{\cdot}c_s^Dds,F_t^D(dv)dt)$ be the predictable characteristics of $D$ under $P$. Then the canonical representation of $D$ relative to a truncation function $h^D(v)$, $v\in\R$, is given by, $P$-a.s., $D=D_0 + B^D + D^c + h^D(v)*(\mu^D-\nu^D)+(v-h^D(v))*\mu^D$. By Proposition 8.26 in \cite{JacodShiryaev}, the exponential specialty of $D$ is equivalent to imposing one of the following conditions:
\begin{itemize}
  \item[{\Ad}] (i) $(e^v-1-h^D(v))*\nu^D\in{\cal V}^P$;~~~(ii)  $e^v{\mathds{1}}_{v>1}*\nu^D\in{\cal V}^P$.
\end{itemize}
}

\section{Power FIPP: Characterization and Optimal Strategy}\label{sec:forwardutility}

This section characterizes a class of power $P$-FIPPs. Section \ref{sec:decomp} develops an explicit multiplicative decomposition of a power random field. Section \ref{sec:conditions} then provides the exact conditions under which the decomposed random field is a power $P$-FIPP. Section \ref{sec:strategy} studies the optimal trading strategy associated with the FIPP.

\subsection{Multiplicative decomposition of a power random field}\label{sec:decomp}

Let $p\in(-\infty,0)\cup(0,1)$ and $D$ be the scalar $P$-semimartingale introduced at the end of Section~\ref{sec:constraints}. Consider a random field $U$ of the form
\begin{align}\label{eq:UtxL}
U_t(x) := \frac{e^{D_t}}{p} x^p,\quad (t,x)\in[0,\infty)\times(0,\infty).
\end{align}
We omit the dependence of $U_t(x)$ on $p$ to lighten notation. It follows from Assumption {\Ad}-(i) or (ii) that $L:=e^D$ is a special $P$-semimartingale that admits the decomposition $L=e^{D_0}+M^L+V^L$, $P$-a.s., where $M^L\in{\cal M}_{loc}^P$ and $V^L\in{\cal V}^P\cap{\cal P}$. We provide more details about this decomposition in {the following} Lemma~\ref{lem:ReptildeMA}.
\begin{lemma}\label{lem:ReptildeMA}
Let Assumption {\Ad}-{\rm(i)} or {\rm(ii)} hold. Then, it holds that, $P$-a.s.
\begin{align}\label{eq:eD-decom}
L=e^{D_0} {\cal E}(M){\cal E}(V).
\end{align}
The pair $(M,V)\in{\cal M}_{loc}^P\times({\cal V}^P\cap{\cal P})$ admits the following representation:
\begin{align}\label{eq:hatAhatM}
M&=D^c+ (e^v-1)*(\mu^D-\nu^D),~~~V= B^D+ \frac{C^{D}}{2}+\{e^v-1-h^D(v)\}*\nu^D,
\end{align}
where the jump process of $M$ is given by $\Delta M=(e^{\Delta D}-1){\I}_{\Delta D\neq0}$, and $V$ is continuous, i.e., $\Delta V=0$.
\end{lemma}

\subsection{Characterization of $U$ as a  $P$-FIPP}\label{sec:conditions}

In this section, we identify the conditions under which the random field $U$ in (\ref{eq:UtxL}) is a $P$-FIPP. The multiplicative decomposition of $U$ given in~\eqref{eq:eD-decom}, under the conditions of Lemma~\ref{lem:ReptildeMA}, yields the following representation: for $(t,x)\in[0,\infty)\times(0,\infty)$,
\begin{align}\label{eq:UtxL2}
U_t(x) = U_0(x){\cal E}_t(M){\cal E}_t(V)={\cal E}_t(M){U}^V_t(x),~~\text{with}~U_0(x)=\frac{1}{p}x^pe^{D_0},
\end{align}
where the random field $U^V$ is defined by
\begin{align}\label{eq:tildeU}
U_t^V(x):=U_0(x){\cal E}_t(V),\quad (t,x)\in[0,\infty)\times(0,\infty).
\end{align}
By Lemma~\ref{lem:ReptildeMA}, we have $M\in{\cal M}_{loc}^{P,+}$, and hence ${\cal E}(M)>0$. Thus, ${\cal E}(M)$ may be viewed as a local density process (see also~\cite{LipShir}, pp. 220-221). Next, we use the decomposition~\eqref{eq:UtxL2} to develop a local change of measure that characterizes $U$ as a $P$-FIPP. For $t\geq0$, define
\begin{align}\label{eq:QM}
\frac{dQ^M}{dP}\bigg|_{\F_t} = {\cal E}_t(M).
\end{align}
\begin{lemma}\label{lem:equifipp}
Assume that the probability measure $Q^M$ defined in \eqref{eq:QM} belongs to ${\cal Q}_M$. Then, the following statements hold:
\begin{itemize}
\item[{\rm(i)}] $\Gamma_{\cal C}^{P}=\Gamma_{\cal C}^{Q^M}$, i.e., the set of admissible strategies under $P$ coincides with the corresponding set under $Q^M$;
\item[{\rm(ii)}] the random field $U$ given by \eqref{eq:UtxL2} is a $P$-FIPP with optimal trading strategy $\pi^*\in\Gamma_{\cal C}^{P}$ if and only if the random field $U^V$ given by \eqref{eq:tildeU} is a $Q^M$-FIPP with  optimal trading strategy $\pi^*\in\Gamma_{\cal C}^{Q^M}$.
\end{itemize}
\end{lemma}
Next, we characterize the $P$-semimartingale $D$ that makes $U$ given by \eqref{eq:UtxL2} a $P$-FIPP with corresponding optimal trading strategy $\pi^*\in\Gamma_{\cal C}^P$. The main result of this section is stated as follows:
\begin{theorem}\label{prop:snforward}
Let Assumption {\Ac} and the conditions of Lemma~\ref{lem:ReptildeMA}-\ref{lem:equifipp} hold. {The random field $U$ given by \eqref{eq:UtxL2} is a $P$-FIPP with optimal trading strategy $\pi^*\in\Gamma_{\cal C}^{P}$} if and only if the following conditions hold, \ASP
\begin{itemize}
\item[{\rm(i)}] ${\cal E}(\beta^M\cdot R^{M,c}+\overline{W}^M(u)*(\mu-\nu^{M})){\cal E}(M)$ is a $P$-martingale, where $R^{M,c}$ and $\nu^M$ are respectively given by \eqref{eq:RMci} and \eqref{eq:bcaR}.
\item[{\rm(ii)}] there exists $\pi^*\in\Gamma_{\cal C}^{P}$ such that $\Phi_p^M(\pi^*)=\underset{\pi\in{\cal C}\cap{\cal C}_0}{\sup}\Phi_p^M(\pi)$;
\item[{\rm(iii)}] The predictable characteristics $(B^D,C^D,\nu^D(dt,dv))=(\int_0^{\cdot}b_s^Dds,\int_0^{\cdot}c_s^{D}ds,F_t^D(dv)dt)$ of $D$ satisfy the equation
\begin{align}\label{eq:optimcond000}
\sup_{\pi\in{\cal C}\cap{\cal C}_0}\Phi_p^M(\pi)+p^{-1}\left\{b^D+ \frac{c^{D}}{2}+(e^v-1-h^D(v))*F^D\right\}=0.
\end{align}
\end{itemize}
For $\pi\in{\cal C}_0$, the random mapping
\begin{align}\label{eq:Phipi000}
\Phi_p^M(\pi):=\pi^{\top}b^{M}+\frac{p-1}{2}\pi^{\top}c\pi+\{p^{-1}(1+\pi^{\top}u)^p-p^{-1}-\pi^{\top}h(u)\}*F^{M},
\end{align}
where $(b^M,F^M)$ is given by \eqref{eq:diffbcaR}, $\beta_t^M:=p\pi_t^*$, and $\overline{W}_t^M(u):=\{1+({\pi}_t^*)^{\top}u\}^{p}-1$.
\end{theorem}

{Observe that the conditions (i)-(iii) in Theorem~\ref{prop:snforward} do not depend on the choice of $D_0$. That is, the initial utility $U_0(x)=e^{D_0}\frac{x^p}{p}$ is compatible with a FIPP of the form given in~\eqref{eq:UtxL2}.} Thanks to Lemma~\ref{lem:equifipp}, the proof of Theorem~\ref{prop:snforward} is reduced to characterizing the exponentially special semimartingale $D$ that makes $U^V$ given by \eqref{eq:tildeU} a $Q^M$-FIPP with optimal trading strategy $\pi^*\in\Gamma_{\cal C}^{Q^M}$. We next state two auxiliary lemmas. For $\pi\in{\Gamma}_{\cal C}^{Q^M}$, and $(\omega,t,u)\in\Omega\times[0,\infty)\times\R^n$, define
\begin{align}\label{eq:HWpi}
W_t^{\pi}(\omega,u) &:=p^{-1}(1+\pi_t(\omega)^{\top}u)^p-p^{-1},
\end{align}
and recall the space $G_{loc}^{Q^M}(\mu)$, related to the jump measure $\mu$ of $R$ under $Q^M$, and specified in Def. II.1.27-(a) of \cite{JacodShiryaev}, page\ 72.
\begin{lemma}\label{lem:sdehatA}
Assume that the conditions of Theorem~\ref{prop:snforward} hold and the process $U^V$ given by \eqref{eq:tildeU} is a $Q^M$-FIPP. Then, for any $\pi\in{\Gamma}_{\cal C}^{Q^M}$, $\pi\in L_{loc}^{Q^M,2}(R^{M,c})$, $W^{\pi}\in G_{loc}^{Q^M}(\mu)$, and $|p^{-1}(1+\pi^{\top}u)^p-p^{-1}-\pi^{\top}h(u)|*\nu^{M}\in{\cal A}_{loc}^{Q^M,+}$. Moreover, the predictable characteristics $(B^D,C^D,\nu^D(dt,dv))=(\int_0^{\cdot} b_s^Dds,\int_0^{\cdot}c_s^{D}ds,F_t^D(dv)dt)$ of $D$ satisfy \eqref{eq:optimcond000}, \ASQ
\end{lemma}

For $\pi\in{\cal C}_0$, the concave function $\Phi_p^M(\pi)$ given by \eqref{eq:Phipi000} is well defined and takes values in $\R\cup\{{\rm sign}(p)\infty\}$, see, e.g., Lemma 5.3 in~\cite{Nutzmf}.

\noindent{\bf Proof of Theorem~\ref{prop:snforward}.}\quad By Lemma~\ref{lem:equifipp}, it is enough to prove that $U^V$ given by \eqref{eq:tildeU} is a $Q^M$-FIPP with the optimal trading strategy $\pi^*\in\Gamma^{Q^M}_{\cal C}$ if and only if (i)-(iii) hold.

$\Longrightarrow$ Suppose that $U^V$ given by \eqref{eq:tildeU} is a $Q^M$-FIPP with the optimal trading strategy $\pi^*\in\Gamma^{Q^M}_{\cal C}$. Then, by Definition~\ref{def:FIPP}, we have $U^V(X^{\pi^*,x})$ is a $Q^M$-martingale for any $x>0$. On the other hand, for any $\pi\in\Gamma^{Q^M}_{\cal C}$, it follows from~\eqref{eq:hatUrep} and~\eqref{eq:cond1122} in Appendix that
\begin{align}\label{eq:UtU0}
U^V(X^{\pi,x})&={U}_0(x){\cal E}\left(V+p\int_0^{\cdot}\Phi_p^M(\pi_s)ds+p\pi\cdot R^{M,c} + pW^{\pi}(u)*(\mu-\nu^{M})\right),
\end{align}
where $W^{\pi}$ have been defined in \eqref{eq:HWpi}. Because $U^V$ is a $Q^M$-FIPP, Lemma~\ref{lem:sdehatA} implies that $p\pi\in L_{loc}^{Q^M,2}(R^{M,c})$, $pW^{\pi}\in G_{loc}^{Q^M}(\mu)$, and the condition (iii) holds. Moreover, using the equality \eqref{eq:optimcond} in the proof of Lemma~\ref{lem:sdehatA}, we also have that
\begin{align}\label{eq:FUEQN}
\frac{dV}{dt}+p\sup_{\pi\in{\cal C}\cap{\cal C}_0}\Phi_p^M(\pi)=0.
\end{align}
In view of \eqref{eq:UtU0}, we have that
\begin{align}\label{eq:UtU0pistar}
U^V(X^{\pi^*,x})={U}_0(x){\cal E}\left(V+p\int_0^{\cdot}\Phi_p^M(\pi_s^*)ds+\beta^M\cdot R^{M,c} + \LW^M(u)*(\mu-\nu^{M})\right)
\end{align}
is a $Q^M$-martingale, it is a local $Q^M$-martingale (clearly, it is a $\sigma$-martingale under $Q^M$). This implies the drift rate (w.r.t. $A_t=t$) of $U^V(X^{\pi^*,x})$ is zero, i.e.,
\begin{align}\label{eq:FUEQN2}
\frac{dV}{dt}+p\Phi_p^M(\pi^*)=0.
\end{align}
Taking the difference between~\eqref{eq:FUEQN} and~\eqref{eq:FUEQN2}, we deduce that
\begin{align*}
p\left(\sup_{\pi\in{\cal C}\cap{\cal C}_0}\Phi_p^M(\pi)-\Phi_p^M(\pi^*)\right)=0,
\end{align*}
which implies that the condition (ii) holds. Moreover, it follows from \eqref{eq:UtU0pistar} and \eqref{eq:FUEQN2} that
\begin{align}\label{eq:UVpistar}
U^V(X^{\pi^*,x})={U}_0(x){\cal E}\left(\beta^M\cdot R^{M,c} + \LW^M(u)*(\mu-\nu^{M})\right),
\end{align}
and hence that condition (i) holds.

$\Longleftarrow$ Suppose  conditions (i)-(iii) hold.  Lemma~\ref{lem:ReptildeMA}, together with~\eqref{eq:optimcond000}, implies that
\begin{align}\label{eq:proofeqn}
\frac{dV}{dt}+p\Phi_p^M({\pi}^*)=0,~~\sup_{\pi\in{\cal C}\cap{\cal C}_0}\Phi_p^M(\pi)=\Phi_p^M({\pi}^*).
\end{align}
Using the equations above, along with~\eqref{eq:UtU0pistar}, we obtain that
\begin{align}\label{eq:UVpistar2}
U^V(X^{\pi^*,x})={U}_0(x){\cal E}\left(\beta^M\cdot R^{M,c} + \LW^M(u)*(\mu-\nu^{M})\right).
\end{align}
Then, condition (i) implies {that} $U^V(X^{\pi^*,x})$ is a $Q^M$-martingale, where $\pi^*\in{\cal C}\cap{\cal C}_0$ is determined by \eqref{eq:proofeqn}.

We next prove that, for any $\pi\in\Gamma_{\cal C}^{Q^M}$, $U^V(X^{\pi,x})$ is a $Q^M$-supermartingale for all $x>0$. We prove this claim by considering the following two subcases:

{\bf(a)} The power parameter $p\in(0,1)$. In this case, it follows from \eqref{eq:proofeqn} that, for any $\pi\in\Gamma_{\cal C}^{Q^M}$,
\begin{align}\label{eq:driftbigger0}
p\Phi_p^M({\pi})\leq p\Phi_p^M({\pi}^*)=-\frac{dV}{dt}.
\end{align}
We next show that $U^V(X^{\pi,x})$ given in~\eqref{eq:UtU0} is a $Q^M$-supermartingale. First, observe that the drift rate (w.r.t. $A_t=t$) of the semimartingale
\begin{align*}
 V+p\int_0^{\cdot}\Phi_p^M(\pi_s)ds+p\pi\cdot R^{M,c} + pW^{\pi}(u)*(\mu-\nu^{M})
\end{align*}
is given by $\frac{dV}{dt}+p\Phi_p^M({\pi})$. It follows from \eqref{eq:driftbigger0} that the drift rate is negative. Recall that, for the process $N(\pi)$ defined in \eqref{eq:tildeN1} of the Appendix, it holds that
\begin{align*}
V+ N(\pi)=V+p\int_0^{\cdot}\Phi_p^M(\pi_s)ds+p\pi\cdot R^{M,c} + pW^{\pi}(u)*(\mu-\nu^{M}).
\end{align*}
It can be seen from \eqref{eq:tildeN1} that, for all $\pi\in\Gamma^{Q^M}_{\cal C}$, $1+\Delta(V+N(\pi))=1+\Delta N(\pi)>0$ since $\Delta V=0$, and hence, for all $\pi\in\Gamma^{Q^M}_{\cal C}$,
\begin{align}\label{eq:se1}
{\cal E}\left(V+p\int_0^{\cdot}\Phi_p^M(\pi_s)ds+p\pi\cdot R^{M,c} + pW^{\pi}(u)*(\mu-\nu^{M})\right)
\end{align}
is strictly positive. It thus follows from the inequality in~\eqref{eq:driftbigger0} that the drift rate (w.r.t. $A_t=t$) of the Dol\'eans-Dade exponential \eqref{eq:se1} is also negative. Using Proposition 11.3 in Appendix 3 of~\cite{Kara-Kard}, we conclude that the Dol\'eans-Dade exponential \eqref{eq:se1} is a $Q^M$-supermartingale, because this stochastic exponential is uniformly bounded from below. This, along with the fact that $U_0(x)>0$, leads to the conclusion that $U^V(X^{\pi,x})$ is a $Q^M$-supermartingale for all $\pi\in\Gamma_{\cal C}^{Q^M}$. Note that $U_t^V(x)$ itself is positive for all $(t,x)\in[0,\infty)\times(0,\infty)$, and hence $(U^V(X^{\pi^*,x}))^-$ is clearly of class (DL). Therefore, the vector $\pi^*\in{\cal C}_0\cap{\cal C}$ specified by \eqref{eq:proofeqn} is the optimal trading strategy.

{\bf(b)} The power parameter $p<0$. In this case, we have that
\begin{align}\label{eq:driftsmaller0}
p\Phi_p^M({\pi})\geq p\Phi_p^M({\pi}^*)=-\frac{dV}{dt}.
\end{align}
We then deduce from \eqref{eq:driftsmaller0} that the drift rate (w.r.t. $A_t=t$) of
\begin{align*}
-{\cal E}\left(V+p\int_0^{\cdot}\Phi_p^M(\pi_s)ds+p\pi\cdot R^{M,c} + pW^{\pi}(u)*(\mu-\nu^{M})\right)
\end{align*}
is negative. Because ${U}_0(x)$ is negative for $x>0$, it follows that, for any $\pi\in\Gamma^{Q^M}_{\cal C}$ and $x>0$, $U^V(X^{\pi,x})=-U^V(x{\cal E}(\pi\cdot R))^-$ is a $Q^M$-semimartingale with a negative drift rate. Because $\pi$ belongs to the class of admissible strategies $\Gamma^{Q^M}_{\cal C}$ given in Definition~\ref{def:control-set}, $U^V(x{\cal E}(\pi\cdot R))^-$ is of class (DL). Then, for any $T>0$, we have that $U_{\cdot\wedge T}^V(X^{\pi,x})$ is of class (D). Because a semimartingale of class (D) with negative drift rate is a supermartingale, we conclude that, for all $\pi\in\Gamma^{Q^M}_{\cal C}$ and $x>0$, $U_{\cdot\wedge T}^V(X^{\pi,x})$ is a $Q^M$-supermartingale, i.e., $\Ex^{Q^M}[U_{s\wedge T}^V(X^{\pi,x})|\F_t]\leq U_{t\wedge T}^V(X^{\pi,x})$ for $0\leq t\leq s<+\infty$. Because $(U_{T\wedge t}^V(X^{\pi,x}))_{T\geq0}$ is of class (D), it follows that $U^V(X^{\pi,x})$ is a $Q^M$-supermartingale by letting $T\to\infty$. This yields the condition (ii) in Definition~\ref{def:FIPP}. It remains to verify that $\pi^*\in{\cal C}\cap{\cal C}_0$ determined by \eqref{eq:proofeqn} is the optimal trading strategy. This is equivalent to proving that $(U^V(X^{\pi^*,x}))^-$ is of class (DL). From \eqref{eq:UVpistar2}, we then obtain that
$(U^V(X^{\pi^*,x}))^-=-{U}_0(x){\cal E}(\beta^M\cdot R^{M,c} + \LW^M(u)*(\mu-\nu^{M}))$
is a $Q^M$-martingale using (i). It can then be easily seen that $(U^V(X^{\pi^*,x}))^-$ is of class (DL). This completes the proof of the theorem. \hfill$\Box$

{\begin{remark}\label{rem:choullima}
Similar to our study, \cite{choulli17} employ a change of measure technique to establish a power forward performance criterion. However, there exist major technical and conceptual differences between our approach and theirs, originating from the different set of assumptions made in the two studies.
\cite{choulli17} assume the discounted risky asset price process to be a {\it locally bounded} semimartingale $S$. The proofs of their Theorem 3.2-3.3 strongly rely on the assumption of {\it local boundedness} of $S$. They use the minimal Hellinger martingale (MHM) densities to mainly characterize sufficient and necessary conditions on a process $D$ used in the construction of the power random field (see their Theorem 3.2-3.3). In addition, their analysis ignores portfolio constraints, and can thus directly use the first-order equation for optimality. It is worth highlighting that the existence of the minimal Hellinger martingale density 
heavily depends on the {\it locally boundedness} assumption of $S$ (plus some other assumptions). 
The locally boundedness of $S$ (which, in particular, implies that $\int|x|^2F(dx)<\infty$ where $F$ is the differential characteristics of the jump compensator of $S$) helps proving fundamental properties such as convexity, and first-order differentiability, of the function $\Phi_p^R(\lambda)$. This in turn makes it possible to use the first-order condition to pin down the optimal trading strategy (see Lemma 6.2 in \cite{choulli17}). Unlike their study, we do not impose that the multi-asset discount price process is {\it locally bounded}. As a result, the proof of our Theorem \ref{prop:snforward} uses a different argument because the MHM density argument is not applicable in our semimartingale market model. 
In Definition~\ref{def:control-set}, apart from the constraint set ${\cal C}$, a key condition is that  ``{$U_{t}^Q(x{\cal E}_{t}(\pi\cdot R))^{-}$, $t\in\R_+$, is of class {\rm(DL)} under $Q$}". This key condition makes it challenging to prove the equivalent characterization of the forward utility under $P$ and $Q$, compared to the proof of Proposition 2.1 in \cite{choulli17}. We prove the equivalence result in Lemma~\ref{lem:equifipp} by establishing an equivalent condition of the class of (DL) (see Lemma~A.1 in the Appendix). Additionally, we provide the equivalent characterization of the optimal trading strategies under $P$ and $Q$ in Lemma~\ref{lem:equifipp}.
\end{remark}}

\subsection{Optimal trading strategies}\label{sec:strategy}

{By} Theorem~\ref{prop:snforward}, the optimal strategy $\pi^*$ under the original probability measure $P$ can be recovered as the solution of the constrained extremum problem, with objective function $\Phi_p^M(\pi)$ defined by \eqref{eq:Phipi000}, and subject to the constraint that $\pi\in{\cal C}\cap{\cal C}_0$. For $\pi\in{\cal C}_0$, the mapping $\Phi_p^M(\pi)$ can be rewritten as:
\begin{align}\label{eq:PhiQpi2}
\Phi_p^M(\pi) &= \pi^{\top}b^{M} +\frac{p-1}{2}\pi^{\top}c\pi+I^M(\pi),
\end{align}
where we choose the truncation function $h(u)=u{\I}_{|u|\leq1}$, and the function
\begin{align}\label{eq:ipim}
I^M(\pi)&:=I_1^M(\pi)+I_2^M(\pi)\\
&=\int_{|u|\leq1}\{p^{-1}(1+\pi^{\top}u)^p-p^{-1}-\pi^{\top}u\}F^{M}(du)+\int_{|u|>1}\{p^{-1}(1+\pi^{\top}u)^p-p^{-1}\}F^{M}(du).\nonumber
\end{align}
Then, under {\Af} in Section~\ref{sec:constraints}, the optimal strategy $\pi^*$ can be completely characterized (see Lemma~A.2 in the Appendix).

If the return process $R$ is a L\'evy process and the constraint set satisfies ${\cal C}_t(\omega)=C$ for all $(t,\omega)$ (where $C\subset\R^n$ satisfies Assumption {\Ac}), the random mapping $\Phi_p^M(\pi)$ becomes deterministic. Note that the budget constraint $C_0:={\cal C}_0$ is also independent of $(t,\omega)$, and hence it is also a (deterministic) closed subset of $\R^n$. We study the constrained extremum problem with objective function $\Phi_p^M(\pi)$, and constraint set $C_0\cap C$. For $\pi \in \R^n$, define the convex function
\begin{align}\label{eq:convexfcng}
\psi(\pi) :=\left\{
          \begin{array}{cl}
            -\pi^{\top}b^{M} -\frac{p-1}{2}\pi^{\top}c\pi-\{p^{-1}(1+\pi^{\top}u)^p-p^{-1}-\pi^{\top}h(u)\}*F^{M}, & \pi\in C_0;\\ \\
           +\infty, & \pi\notin C_0.
          \end{array}
        \right.
\end{align}
We recall that $(b^M,F^M)$ is given by \eqref{eq:diffbcaR}. {Hence}, the constrained extremum problem with objective function $\Phi_p^M(\pi)$ and constraint set $C_0\cap C$ is equivalent to solving the unconstrained convex minimization problem
\begin{align}\label{eq:convexf}
\inf_{\pi\in\R^n} \left \{ \psi(\pi) + \delta(\pi|C) \right \},
\end{align}
where $\delta(\cdot|C)$ denotes the indicator function of the convex set $C$, i.e., it equals $0$ if $\pi\in C$, and $+\infty$ otherwise. The indicator function $\delta(\cdot|C)$ is clearly convex. In view of \eqref{eq:convexfcng}, and noting that $C\cap C_0$ contains the origin, this implies that $C\cap{\rm dom}(\psi)\neq\emptyset$, where ${\rm dom}(\psi):=\{\pi\in\R^n;\ \psi(\pi)<+\infty\}$ is the efficient domain of $\psi$. The condition~\eqref{eq:xpintegrableFR} also implies that $\psi(\pi)$ is a closed, proper and convex function on $\R^n$.

The following lemma provides explicit representations of the recession function $\psi0^+$ of $\psi$, and of the recession cone of $\psi$. First, we introduce the following sets:
\begin{align}\label{eq:Lambda}
\Lambda_{-}(y):=\{u\in\R_0^n;\ y^{\top}u<0\},\qquad \Lambda^+(y):=\{u\in\R_0^n;\ y^{\top}u>0\}.
\end{align}
\vspace{-1cm}
\begin{lemma}\label{lem:recessionfcn}
Under Assumption {\Af}, it holds that
\begin{itemize}
\item[{\rm(i)}] If $\pi\in\{y\in\R^n;\ F^{M}(\Lambda_{-}(y))=0\}$, then
\begin{align}\label{eq:g0+case1}
\psi0^+(\pi)=\left\{
            \begin{array}{cl}
         \displaystyle     \pi^{\top}(h(u)*F^{M}-b^{M}), & c\pi=0;\\[0.6em]
         \displaystyle     +\infty, & c\pi\neq0.
            \end{array}
          \right.
\end{align}
If $\pi\in\{y\in\R^n;\ F^{M}(\Lambda_{-}(y))>0\}$, then $\psi0^+(\pi)=+\infty$.
\item[{\rm(ii)}] If $\pi\in\{y\in\R^n;\ F^{M}(\Lambda^{+}(y))=0\}$, then
\begin{align}\label{eq:g0+case2}
\psi0^+(-\pi)=\left\{
            \begin{array}{cl}
           \displaystyle   -\pi^{\top}(h(u)*F^{M}-b^{M}), & c\pi=0;\\[0.6em]
           \displaystyle   +\infty, & c\pi\neq0.
            \end{array}
          \right.
\end{align}
If $\pi\in\{y\in\R^n;\ F^{M}(\Lambda^{+}(y))>0\}$, then $\psi0^+(-\pi)=+\infty$.
\item[{\rm(iii)}] The recession cone of $\psi$ is given by
\begin{align}\label{eq:recessionconeg}
\left\{\pi\in\R^n;\ F^{M}(\Lambda_{-}(\pi))=0,\ c\pi=0,\ \pi^{\top}(h(u)*F^{M}-b^{M})\leq0\right\}.
\end{align}
\item[{\rm(iv)}]The constancy space of $\psi$ is given by ${\cal N}^{Q^M}$.
\end{itemize}
\end{lemma}

{By the claim (iv) of Lemma~\ref{lem:recessionfcn}, ${\cal N}^{Q^M}$ is the largest subspace contained in the recession cone of $\psi$. The directions of the vectors in ${\cal N}^{Q^M}$ are the directions along which $\psi$ is constant. \cite{Kardaras09mf} refers to every element in the set $\{\pi\in\R^n;\ \psi0^+(\pi)\leq0\}\setminus{\cal N}^{Q^M}$
as an Immediate Arbitrage Opportunity under $Q^M$. Essentially, those are constant portfolios that result in increasing profits. The following result is a consequence of Lemma~\ref{lem:recessionfcn}.}
\begin{proposition}\label{thm:optimum}
Let assumptions {\Ac} and {\Af} hold. If, it holds that
\begin{align}\label{eq:cond11}
\left\{\pi\in\R^n;\ F^{M}(\Lambda_{-}(\pi))=0,\ c\pi=0,\ \pi^{\top}(h(u)*F^{M}-b^{M})\leq0\right\}\cap0^+ C=\{0\},
\end{align}
then $\psi$ attains its infimum over $C$. In the case that $C$ is additionally polyhedral, 
if for any nonzero
\begin{align}\label{eq:cond22}
y\in\left\{\pi\in\R^n;\ F^{M}(\Lambda_{-}(\pi))=0,\ c\pi=0,\ \pi^{\top}(h(u)*F^{M}-b^{M})\leq0\right\}\cap0^+C{,}
\end{align}
it holds that $y\in{\cal N}^{Q^M}$, then $\psi$ achieves its infimum over $C$.
\end{proposition}

\noindent{\bf Proof.}\quad
By Theorem 27.3 in~\cite{Rockafellar}, if $\psi$ and $C$ have no direction of recession in common, then $\psi$ attains its infimum over $C$. Using (iii) of Lemma~\ref{lem:recessionfcn}, this condition in our case is equivalent to the condition \eqref{eq:cond11}. In the case where $C$ is also polyhedral, the condition that for any nonzero element in the set \eqref{eq:cond22}, we have that $y\in{\cal N}^{Q^M}$ is equivalent to the statement that every common direction of recession of $\psi$ and $C$ is also a direction in which $\psi$ is constant, using (iv) of Lemma~\ref{lem:recessionfcn}. Hence, $\psi$ achieves its infimum over $C$ by applying the last conclusive statement in Theorem 27.3 of~\cite{Rockafellar}. \hfill$\Box$

\section{Semimartingale Factor Model}\label{sec:BSDE}

Leveraging the multiplicative decomposition of a power random field given in Section~\ref{sec:decomp}, and building upon the equivalent characterization of a power FIPP given by Theorem~\ref{prop:snforward}, we study a class of random fields of the form
\begin{align}\label{eq:factorU}
U(x)=U_0(x){\cal E}(M){\cal E}(V^M),~~\text{with}~U_0(x)=\frac{1}{p}x^pe^{D_0},\quad x>0,
\end{align}
{where $D_0\in\R$ is an input which determines the initial utility}, $M\in{\cal M}_{loc}^{P,+}$ is determined by a semimartingale factor process $Y$, and $V^M\in{\cal V}^P\cap{\cal P}$ is jointly determined by the semimartingale pair $(R,Y)$ (see \eqref{eq:optimcondQn22} below). We first provide conditions for the existence of $U$ as a $P$-FIPP using Theorem~\ref{prop:snforward}. We then relate the performance process $U$ {to a triplet of processes (c.f. \eqref{eq:BSDE-sol} below), whose first component admits an integral representation w.r.t the semimartingale factor $Y$}. 

We begin by introducing the factor market model $(R,Y)$, where the stochastic factor $Y$ is an $\R^d$-valued c\`{a}dl\`ag semimartingale under $P$. Let $(B^Y,C^Y,\nu^Y(dt,dv))=(\int_0^{\cdot}b_s^Yds,\int_0^{\cdot}c_s^Yds,F_t^Y(dv)dt)$ be the predictable characteristics (and hence $(b^Y,c^Y,F^Y)$ is the differential characteristics) of $Y$, relative to a truncation function $h^Y:\R^d\to\R^d$, and $\mu^Y$ be the jump measure of $Y$. Let $(\overline{B},\overline{C},\overline{\nu})$ be the joint predictable characteristics of the $\R^{n+d}$-valued semimartingale $(R,Y)$. Then, the canonical representation of $(R,Y)$ is given by, $P$-a.s.
\begin{align}\label{eq:joint-chara0}
\left(\begin{array}{c}
R \\
Y
\end{array}\right) = \left(\begin{array}{c}
0 \\
Y_0
\end{array}\right)+\left(\begin{array}{c}
B \\
B^Y
\end{array}\right)+\left(
\begin{array}{c}
R^c \\
Y^c
\end{array}\right)+\left(\begin{array}{c}
h(u)\\
h^Y(v)
\end{array}\right)*(\overline{\mu}-\overline{\nu})+\left(\begin{array}{c}
u-h(u) \\
v-h^Y(v)
\end{array}\right)*\overline{\mu},
\end{align}
where $\overline{\mu}$ is the jump measure for $(R,Y)$, and it holds that
\begin{align*}
\overline{B}=\left(\begin{array}{c}
B \\
B^Y
\end{array}\right),~\overline{C}^{ij}=\left\lc\left(
\begin{array}{c}
R^c \\
Y^c
\end{array}\right)^i, \left(
\begin{array}{c}
R^c \\
Y^c
\end{array}\right)^j\right\rc,~\mu(dt,du)=\overline{\mu}(dt,du\times\R_{0}^{d}),~ \mu^Y(dt,dv)=\overline{\mu}(dt,\R_0^{n}\times dv).
\end{align*}

\subsection{Existence of $P$-FIPP $U$ in factor form}\label{sec:FIPPSF}

This section provides the necessary and sufficient condition for the existence of $U$ as a $P$-FIPP in the factor market $(R,Y)$. To start with, we apply the Jacod's decomposition of $M\in{\cal M}_{loc}^{P,+}$ with respect to the factor $Y$,  recalled in the lemma below.
\begin{lemma}[Theorem 3.75 in~\cite{Jacod}]\label{lem:Jacoddecom}
There exists a predictable $Y^c$-integrable process ${H}^M$, $N^M\in{\cal M}_{loc}^P$ satisfying $[Y,N^M]^P=0$, and ${U}^M\in\tilde{\cal P}$, ${G}^M\in\tilde{\cal O}$ satisfying $(|{U}^M|*\nu^Y)_{\infty}<+\infty$, $\sqrt{\sum_{0<s\leq \cdot}{G}^M(s,\Delta Y_s)^2{\I}_{\Delta Y_s\neq0}}\in{\cal A}_{loc}^{P,+}$ and $M_{\mu^Y}^P(G^M|\tilde{\cal P})=0$ such that, $P$-a.s.
\begin{align}\label{eq:logLrep2}
	M=H^M\cdot Y^c + J^M(v)*(\mu^Y-\nu^Y) +{G}^M(v)*\mu^Y + N^M.
\end{align}
Here $J^M:=U^M+\frac{\hat{U}^M}{1-a^Y}$ with $\hat{U}^M:=\int U^M(v)\nu^Y(\{\cdot\},dv)$, and $U^M$ has a version such that $\{a^Y:=\int\nu^Y(\{\cdot\},dv)=1\}\subset\{\hat{U}^M=0\}$. Moreover $\{\Delta N^M\neq0\}\subset\{\Delta Y=0\}$.
\end{lemma}
We refer to $(H^M,U^M,G^M,N^M)$ as Jacod's decomposition of $M$ with respect to the factor $Y$.  Some immediate implications of this  decomposition are presented in the remark below:
\begin{remark}\label{rem:Jacoddecom}
Note that the factor semimartingale $Y$ is QLC, and we have that $U^M\in G_{loc}^P(\mu^Y)$ and $\hat{U}^M=0$. This yields that $J^M=U^M$. Moreover, it holds that $1+\Delta{M} =1+ U^M+G^M$, $M_{\mu^Y}^P$-a.s. Observe that $1+\Delta M>0$, and hence $M_{\mu^Y}^P$-a.s.
\begin{align}\label{eq:1+U+V}
	1+\Delta M=1+ {U}^M+{G}^M >0.
\end{align}
It follows from $U^M\in\tilde{\cal P}$ and $M_{\mu^Y}^P(G^M|\tilde{\cal P})=0$ that $1+ {U}^M>0$, $M_{\mu^Y}^P$-a.s.
\end{remark}

We next introduce processes that will be later used to construct the forward performance process:
\begin{align}\label{eq:BSDE-sol}
\begin{cases}
\displaystyle \Pi^M := \ln({\cal E}({M}){\cal E}(V^M)),\quad  Z^M := H^M; \\[0.5em]
\displaystyle  W^M(v):=M_{\mu^Y}^P(\ln(1+U^M(v)+G^M(v))|\tilde{\cal P});\\[0.5em]
\displaystyle  K^M(v):=\ln(1+U^M(v)+G^M(v))-M_{\mu^Y}^P(\ln(1+U^M(v)+G^M(v))|\tilde{\cal P}).
\end{cases}
\end{align}
The process $V^M\in{\cal V}^P\cap{\cal P}$ satisfies the following (stochastic) equation:
\begin{align}\label{eq:optimcondQn22}
p^{-1}\frac{dV^M}{dt}+\Phi_p^{M}({\pi}^*)=0,\qquad \Phi_p^{M}({\pi}^*)=\sup_{\pi\in{\cal C}_0\cap{\cal C}}\Phi_p^{M}(\pi).
\end{align}
For $\pi\in{\cal C}_0$, we define the random mapping as
\begin{align}\label{eq:PhiM}
\Phi^{M}_p({\pi})&:=\pi^{\top}\{b +cH + h(u)M_{\mu}^P(\Delta{M}|\tilde{\cal P})(u)*F\}+\frac{p-1}{2}\pi^{\top}c\pi\nonumber\\
&\quad+\{p^{-1}(1+\pi^{\top}u)^p-p^{-1}-\pi^{\top}h(u)\}(1+M_{\mu}^P(\Delta{M}|\tilde{\cal P})(u))*F,
\end{align}
where the process $H\in L_{loc}^{P,2}(R^c)$ is determined by the equation
\begin{align}\label{eq:H2}
[M^c,{M}^c]^P=\int_0^{\cdot}H_s^{\top}c_sH_sds.
\end{align}

Throughout the section, we consider a random field of the form \eqref{eq:factorU}, i.e.,
\begin{align}\label{eq:UtxPiM}
	U_t(x)=U_0(x){\cal E}_t(M){\cal E}_t(V^M)=U_0(x)\exp(\Pi_t^M),\quad (t,x)\in[0,\infty)\times(0,\infty).
\end{align}
It follows from \eqref{eq:UtxPiM} that $\Pi_0^M=0$.  The subsequent lemma characterizes $U$ as a $P$-FIPP, and its proof invokes Theorem~\ref{prop:snforward}.
\begin{lemma}\label{lem:BSDE-utility}
Fix ${M}\in{\cal M}_{loc}^{P,+}$. Let $Q^M\in{\cal Q}_M$ be defined by \eqref{eq:QM} and Assumption {\Ac} hold. Then, the random field $U$ given by \eqref{eq:UtxPiM} is a $P$-FIPP with the optimal trading strategy $\pi^*\in{\Gamma}_{\cal C}^{P}$ if and only if the following conditions hold:
\begin{itemize}
\item[{\rm(i)}] there exists $\pi^*\in\Gamma_{\cal C}^{P}$ such that $\sup_{\pi\in{\cal C}\cap{\cal C}_0}\Phi_p^M(\pi)=\Phi_p^M(\pi^*)$, where $\Phi_p^M(\pi)$ is defined by \eqref{eq:PhiM};
\item[{\rm(ii)}] ${\cal E}(\beta^M\cdot R^{M,c}+\LW^M(u)*(\mu-\nu^{M})){\cal E}(M)$ is a $P$-martingale, where $R^{M,c}=R^c-\int_0^{\cdot}c_sH_sds$, i.e., it is the continuous local martingale part of $R$ under $Q^M$, and
\begin{align*}
{\beta}_t^M:=p\pi_t^*,~~\LW_t^M(u):=\{1+(\pi_t^*)^{\top}u\}^p-1,~~\nu^{M}:=(1+M_{\mu}^P(\Delta{M}|\tilde{\cal P})(u))*\nu.
\end{align*}
\end{itemize}
\end{lemma}

The $P$-dynamics of the random field $U$ given by \eqref{eq:UtxPiM} plays an important role in the analysis of the so-called time-monotone processes. An explicit expression for this dynamics follows from \eqref{eq:BSDE-sol} and \eqref{eq:UtxPiM}:
\begin{lemma}\label{lem:dynamicsU}
For any $x>0$, the $P$-dynamics of the random field $U$ defined by \eqref{eq:UtxPiM} is given by
\begin{align}\label{eq:dynamicsU}
dU_t(x) &= U_{t-}(x) \left\{dM_t-p\sup_{\pi\in{\cal C}_0\cap{\cal C}}\Phi_{p}^M(\pi)dt\right\},\quad U_{0-}(x)=U_{0}(x)=\frac{x^p}{p}e^{D_0}.
\end{align}
\end{lemma}

Using Lemma~\ref{lem:dynamicsU}, we revisit  time-monotone processes, first introduced by \cite{MusiZari10a} in Brownian markets, and discuss how existing results can be generalized to semimartingale markets:
\begin{example}[Time-Monotone Process with Semimartingale Factor Market]\label{exam:time-monotoneproc}
Consider a semimartingale process $Y$ satisfying $Y^c\equiv0$ and $\mu^Y\equiv0$, i.e., $Y\in{\cal P}\cap{\cal V}^P$ is of the form $Y=Y_0+B^Y=Y_0+\int_0^{\cdot}b^Y(Y_{s})ds$, where $b^Y$ is Lipschitz on $\R^d$. Assume the $P$-predictable characteristics of $R$ is given by $(B,C,\nu(dt,du))=(\int_0^{\cdot}b(Y_{s})ds,\int_0^{\cdot}c(Y_{s})ds,F(Y_{t},du)dt)$, and the constraint ${\cal C}_t(\omega)=C(Y_{t-}(\omega))$ satisfies Assumption {\Ac}. The triple $(b,c,F)$ satisfies conditions (i) and (ii) of Theorem III.2.32 in \cite{JacodShiryaev}. Taking the trivial $P$-local martingale $M\equiv 0$, we obtain $\Pi^0=\Pi_0^0-\int_0^{\cdot}f(Y_{s})ds$. Hence, the dynamics \eqref{eq:dynamicsU} reduces to
\begin{align}\label{eq:fippsde2}
	dU_t(x) &=U_{t-}(x)d\Pi_t^0= -U_{t}(x)f(Y_{t})dt,
\end{align}
where, for $y\in\R^d$, the function
\begin{align*}
	f(y) &:=p\sup_{\pi\in{C}_0(y)\cap{C(y)}}\left\{\pi^{\top}b(y)+\frac{p-1}{2}\pi^{\top}c(y)\pi+\int\{p^{-1}(1+\pi^{\top}u)^p-p^{-1}
-\pi^{\top}h(u)\}F(y,du)\right\},
\end{align*}
with $C_0(y):=\{\pi\in\R^n;\ F(y,\{u\in\R_0^n;~\pi^{\top}u+1<0\})=0\}$. Since $0\in C(y)$, ${\rm sign}(p)f(y)\geq0$ for all $y\in\R^d$. Then, Eq.~\eqref{eq:fippsde2} may be further simplified to
	\begin{align}\label{eq:utxsimple}
	dU_t(x)=\frac{1}{2}\frac{2(1-p)}{p}f(Y_t)\frac{|\partial_xU_t(x)|^2}{\partial_{xx}^2U_t(x)}dt.
	\end{align}
	Note that $\frac{2(1-p)}{p}f(y)\geq0$ because $1-p>0$ for all $y\in\R^d$. Then \eqref{eq:utxsimple} takes a similar form as SPDE (28) with zero volatility in \cite{MusiZari10}, Section 6.1. It was shown in \cite{MusiZari10a} that its solution is characterized by the so-called time-monotone process given by
	\begin{align}\label{eq:solmono}
	U_t(x) = G\left(x,\int_0^t\frac{2(1-p)}{p}f(Y_{s})ds\right)=G\left(x,-\frac{2(1-p)}{p}\Pi_t^0\right),
	\end{align}
	where, for $(x,t)\in(0,\infty)\times[0,\infty)$, $G(x,t)$ solves the fully nonlinear equation $\partial_tG=\frac{|\partial_x G|^2}{2\partial_{xx}^2G}$. The solution of this equation has been studied by \cite{MusiZari10a}.
	In \cite{MusiZari10}, and \cite{NadTeh17}, the stock return process $R$ is a drifted Brownian motion in the factor form (i.e., jumps are not allowed). This is equivalent to assuming that the stock price process is a geometric Brownian motion in the factor form. 
Therefore, Eq.~\eqref{eq:solmono} suggests that the solution of SPDE (28) in \cite{MusiZari10} corresponding to the zero volatility case can also be characterized by a time-monotone process 
in a {\it semimartingale market} (i.e., when the stock return process $R$ is a semimartingale).
\end{example}

\subsection{Integral representation of $\Pi^M$}\label{sec:quadbsde}

This section shows that the processes defined in \eqref{eq:BSDE-sol} {admits an integral representation w.r.t. the semimartingale factor $Y$. This, in turn, serves to construct the factor representation of the processes defined in \eqref{eq:BSDE-sol}.} To establish this representation, we  restrict the class of martingales from ${\cal M}_{loc}^{P,+}$ to a subset $\bar{\cal M}_{loc}^{P,+}$ of ${\cal M}_{loc}^{P,+}$, defined as follows:
\begin{definition}[Subset $\bar{\cal M}_{loc}^{P,+}$ of ${\cal M}_{loc}^{P,+}$]\label{def:MPplus}
	For $\epsilon\in(0,1)$ and $p\geq1$, define the positive increasing process	as follows:
    \begin{align}\label{eq:Thetapeps}
	\Theta^p(\epsilon):=\sum_{0<s\leq\cdot}\frac{|\Delta M_s|^p}{(1+\Delta M_s)^p}\mathds{1}_{\Delta M_s\in(-1,-\epsilon)},
	\end{align}
	{related to the jumps of $M\in{\cal M}_{loc}^{P,+}$.} The space $\bar{\cal M}_{loc}^{P,+}$ is the set of all scalar local martingales $M\in{\cal M}_{loc}^{P,+}$ for which there exists a
	constant $\epsilon\in[0,1)$ such that {\rm(i)} if $M\in{\cal A}_{loc}^P$, then $\Theta^1(\epsilon)\in{\cal A}_{loc}^{P,+}$; {\rm(ii)} if $M\in{\cal H}_{loc}^{P,2}$, then $\Theta^2(\epsilon)\in{\cal A}_{loc}^{P,+}$.
\end{definition}

Obviously, any continuous scalar $P$-local martingale belongs to $\bar{\cal M}_{loc}^{P,+}$. The following theorem {establishes an integral representation of the process $\Pi^M$, which will be used in the construction of the $P$-FIPP:}
\begin{theorem}\label{thm:existence}
Let $M\in\bar{\cal M}_{loc}^{P,+}$ and $(H^M,U^M,G^M,N^M)$ be Jacod's decomposition of $M$ w.r.t. the semimartingale factor $Y$. Consider the process $(\Pi^M,Z^M,W^M,K^M)$ defined by \eqref{eq:BSDE-sol}. We then have that
\begin{itemize}
\item[{\rm(i)}] the predictable process $Z^M\in L_{loc}^{P,2}(Y^c)$, $W^M\in G_{loc}^P(\mu^Y)$, $K^M\in\tilde{\cal O}$ satisfy $M_{\mu^Y}^P(K^M|\tilde{\cal P})=0$ and $\sqrt{\sum_{0<s\leq\cdot}{K}^M(s,\Delta Y_s)^2\mathds{1}_{\Delta Y_s\neq0}}\in{\cal A}_{loc}^{P,+}$;
\item[{\rm(ii)}] the process $N^M\in{\cal M}_{loc}^P$ satisfies $[Y,N^M]^P=0$ and $\{\Delta N^M\neq0\}\subset\{\Delta Y=0\}$;
\item[{\rm(iii)}] the scalar process ${\Pi^M}$ is an exponentially special semimartingale admitting the decomposition $e^{\Pi^M}=M^{\Pi}+V^{\Pi}$,
where $M^{\Pi}\in{\cal M}^P_{loc}$ and $V^{\Pi}\in{\cal V}^P\cap{\cal P}$ is such that $e^{-\Pi_{-}^M}\Delta V^{\Pi}>-1$, $P$-a.s.;
\item[{\rm(iv)}] $(\Pi^M,Z^M,W^M,K^M,N^M)$ satisfies the following equation: for any $0\leq t\leq T<\infty$, $P$-a.s.
\begin{align}\label{eq:bsde-intgral}
		\Pi^M_t &= \Pi^M_T + \int_t^T f(M,Z_{s}^M,W_{s}^M,K_{s}^M)ds - \int_t^T (Z_s^M)^{\top}dY_s^c\\
		&\quad-\int_t^T\int W_s^M(v)(\mu^Y-\nu^Y)(ds,dv)-\int_t^T\int K_s^M(v)\mu^Y(ds,dv)-\int_t^T dN_s^M.\nonumber
\end{align}
\end{itemize}
The random function $f$ is given by
\begin{align}\label{eq:driverf0}
	f(M,Z,W,K)
		&:=\frac{1}{2}Z^{\top}c^{Y}Z+p\sup_{\pi\in{\cal C}_0\cap{\cal C}}\Phi_p^{M}(\pi)-\left\{W(v)+1-e^{W(v)}M_{\mu^Y}^P\left(e^{K(v)}\big|\tilde{\cal P}\right)\right\}*F^Y.
\end{align}
\end{theorem}

Before presenting the proof of Theorem~\ref{thm:existence}, we collect few observations in the following remark:
\vspace{-1cm}
\begin{remark}\label{rem:bsdesol}
Observe that $(R,Y)$ is QLC. We can then make the following claims:
	\begin{itemize}
		\item[{\rm(i)}] As a consequence of from~\eqref{eq:optimcondQn22}, we have that the predictable process of finite variation $V^M$ is continuous. Moreover, the dynamics of $V^M$ is given by
		\begin{align}\label{eq:optimcondQn2233}
		dV_t^M=-p\sup_{\pi\in{\cal C}_0\cap{\cal C}}\Phi_p^{M}(\pi_t)dt.
		\end{align}
		\item[{\rm(ii)}] Let ${M}\in\bar{\cal M}_{loc}^{P,+}$, and $(H^M,U^M,0,0)$ be the  Jacod's decomposition of $M$ w.r.t. $Y$. A direct implication from (iv) in Theorem~\ref{thm:existence} is that the process $(\Pi^M,Z^M,W^M)$ given in \eqref{eq:BSDE-sol} satisfies the followint integral equation: for $0\leq t\leq T<\infty$, $P$-a.s.
		\begin{align}\label{eq:BSDEspecial}
		\Pi^M_t &= \Pi^M_T + \int_t^T f(M,Z_{s}^M,W_{s}^M)ds - \int_t^T (Z_s^M)^{\top}dY_s^c-\int_t^T\int W_s^M(v)(\mu^Y-\nu^Y)(ds,dv).
		\end{align}
		In the expression above, the function $f$ is given by
		\begin{align}\label{eq:driverfspec}
		f(M,Z,W)=\frac{1}{2}Z^{\top}c^{Y}Z+p\sup_{\pi\in{\cal C}_0\cap{\cal C}}\Phi_p^{M}(\pi)-\left\{W(v)+1-e^{W(v)}\right\}*F^Y.
		\end{align}
	\end{itemize}
\end{remark}

\noindent{\bf Proof of Theorem \ref{thm:existence}.}\quad Claim (ii) follows from Lemma~\ref{lem:Jacoddecom}. We next verify Claim (iv). Since $V^M\in{\cal V}^P\cap{\cal P}${, and} $V^M$ is continuous by \eqref{eq:optimcondQn22}, using Yor's formula (see, e.g.~\cite{Kara-Kard}) we obtain that ${\cal E}(M){\cal E}(V^M)={\cal E}(V^M+ M)$. An application of It\^o's formula yields
	\begin{align}\label{eq:dpi0}
	d\Pi^M &= dV^M+dM-\frac{1}{2}d\left\lc {M}^c,{M}^c\right\rc^P+d\sum\{\ln(1+\Delta{M})-\Delta{M}\}.
	\end{align}
It follows from \eqref{eq:optimcondQn22} that $dV^M=-p \Phi_p^M(\pi^*)dt=-p\sup_{\pi\in{\cal C}_0\cap{\cal C}} \Phi_p^M(\pi)dt$.
	Using \eqref{eq:logLrep2}, Eq.~\eqref{eq:dpi0} can be rewritten as:
	\begin{align}\label{eq:dpi}
	d\Pi^M &= -\left(p\sup_{\pi\in{\cal C}_0\cap{\cal C}} \Phi_p^M(\pi)+\frac{1}{2}(H^M)^{\top}c^{Y}H^M\right)dt+ H^MdY^c\nonumber\\
	&\quad +  U^M(v)*d(\mu^Y-\nu^Y)+ V^G(v)*d\mu^Y+dN^M\nonumber\\
	&\quad+\{\ln(1+U^M(v)+G^M(v))-U^M(v)-G^M(v)\}*d\mu^Y.
	\end{align}
Then, the predictable compensator of $\{\ln(1+U^M(v)+G^M(v))-U^M(v)-G^M(v)\}*\mu^Y$ is given by $\{W^M(v)-U^M(v)\}*\nu^Y$, where we used the fact that $M_{\mu^Y}^P({G}^M|\tilde{\cal P})=0$ and $W^M(v)$ is defined in \eqref{eq:BSDE-sol}. It can be seen that $W^M\in\tilde{\cal P}$. Using the above arguments, we can rewrite~\eqref{eq:dpi} as
	\begin{align}\label{eq:dpi2}
	d\Pi^M &= -\left(p\sup_{\pi\in{\cal C}_0\cap{\cal C}} \Phi_p^M(\pi)+\frac{1}{2}(H^M)^{\top}c^{Y}H^M-\{W^M(v)-U^M(v)\}*F^Y\right)dt \nonumber\\
	&\quad+ H^MdY^c+W^M(v)*d(\mu^Y-\nu^Y)+K^M(v)*d\mu^Y+dN^M,
	\end{align}
	where we recall the expression of $K^M(v)$ given in~\eqref{eq:BSDE-sol}. Notice that $W^M(v)+K^M(v)=\ln(1+U^M(v)+G^M(v))$, and hence $U^M(v) = e^{W^M(v)+K^M(v)}-1-G^M(v)$.
	Observe that $U^M\in\tilde{\cal P}$ and $W^M\in\tilde{\cal P}$. Applying the operator $M_{\mu^Y}^P(\cdot|\tilde{\cal P})$ to both sides of the above equation, we obtain that
	\begin{align}\label{eq:UMv}
	U^M(v) = e^{W^M(v)}M_{\mu^Y}^P\left(e^{K^M(v)}\big|\tilde{\cal P}\right)-1.
	\end{align}
Hence, using \eqref{eq:driverf0} we deduce that
	\begin{align}\label{eq:1stdPIM}
	&p\sup_{\pi\in{\cal C}_0\cap{\cal C}} \Phi_p^M(\pi)+\frac{1}{2}(H^M)^{\top}c^{Y}H^M-\{W^M(v)-U^M(v)\}*F^Y\nonumber\\
	&\quad= p\sup_{\pi\in{\cal C}_0\cap{\cal C}} \Phi_p^M(\pi)+\frac{1}{2}(H^M)^{\top}c^{Y}H^M-\left\{W^M(v)-e^{W^M(v)}M_{\mu^Y}^P\left(e^{K^M(v)}\big|\tilde{\cal P}\right)+1\right\}*F^Y\nonumber\\
	&\quad=f(M,H^M,W^M,K^M).
	\end{align}
	Note that $Z^M=H^M$ in~\eqref{eq:BSDE-sol}. It then follows from \eqref{eq:dpi2} that
	\begin{align}\label{eq:BSDE2222}
	d\Pi^M &= -f(M,Z^M,W^M,K^M)dt+ Z^MdY^c+W^M(v)*d(\mu^Y-\nu^Y)+K^M(v)*d\mu^Y+dN^M.
	\end{align}
	This yields {Claim} (iv).
	
	We next prove {Claim} (i). We first verify that $W^M\in\tilde{\cal P}$ also belongs to $G_{loc}^P(\mu^Y)$ (see also~\eqref{eq:BSDE-sol}). {Note that $Y$ is QLC.} It then follows from II.1.31 and II.1.32 in \cite{JacodShiryaev} that $C(W^M)=|W^M|^2*\nu^Y$ and $\bar{C}(W^M)=|W^M|*\nu^Y$. Observe that $M\in\bar{\cal M}_{loc}^{P,+}$ and by Proposition~I.4.17 in \cite{JacodShiryaev}, up to a localization, it is enough to prove that $W^M\in G_{loc}^P(\mu^Y)$ both for the case of $M\in{\cal M}_{loc}^P\cap{\cal H}^{P,2}$ satisfying $\Theta^2(\epsilon)\in{\cal A}^{P,+}$ for some $\epsilon\in[0,1)$, and for the case $M\in{\cal M}_{loc}^P\cap{\cal A}^P$ satisfying $\Theta^1(\epsilon)\in{\cal A}^{P,+}$ for some $\epsilon\in[0,1)$. We first consider the case  $M\in{\cal M}_{loc}^P\cap{\cal H}^{P,2}$ satisfying $\Theta^2(\epsilon)\in{\cal A}^{P,+}$ for some $\epsilon\in[0,1)$. It then follows from \eqref{eq:1+U+V} and Jensen's inequality applied with the operator $M_{\mu^Y}^P(\cdot|\tilde{\cal P})$ (see Problem 3.2.11 in~\cite{LipShir}) that
	\begin{align}\label{eq:ECLambda}
	\Ex^P\left[C(W^M)_{\infty}\right]&= \Ex^P\left[(|W^M|^2*\nu)_{\infty}\right]=\Ex^P\left[(|W^M|^2*\mu)_{\infty}\right]=M_{\mu^Y}^P\left(|W^M|^2\right)\nonumber\\
	&\leq M_{\mu^Y}^P\left(|\ln(1+\Delta M)|^2\right)=\Ex^P\left[\sum_{s>0}|\ln(1+\Delta M_s)|^2\mathds{1}_{\Delta M_s\neq0}\right]\\
	&=\Ex^P\left[\sum_{s>0}|\ln(1+\Delta M_s)|^2\mathds{1}_{\Delta M_s\in[-\epsilon,\infty)}\right]+\Ex^P\left[\sum_{s>0}|\ln(1+\Delta M_s)|^2\mathds{1}_{\Delta M_s\in(-1,-\epsilon)}\right].\nonumber
	\end{align}
	Using the inequality $\frac{x}{1+x}\leq\ln(1+x)\leq x$ for all $x>-1$, it holds that $|\ln(1+x)|\leq\max\{|x|,\frac{|x|}{1+x}\}$ for all $x>-1$. Then, it holds that
	\begin{align*}
	&\Ex^P\left[\sum_{s>0}|\ln(1+\Delta M_s)|^2\mathds{1}_{\Delta M_s\in[-\epsilon,\infty)}\right]\leq \Ex^P\left[\sum_{s>0}\max\left\{|\Delta M_s|^2,\frac{|\Delta M_s|^2}{(1+\Delta M_s)^2}\right\}\mathds{1}_{\Delta M_s\in[-\epsilon,\infty)}\right]\nonumber\\
	&\qquad\leq\Ex^P\left[\sum_{s>0}|\Delta M_s|^2\right]+\Ex^P\left[\sum_{s>0}\frac{|\Delta M_s|^2}{(1+\Delta M_s)^2}\mathds{1}_{\Delta M_s\in[-\epsilon,\infty)}\right]\nonumber\\
	&\qquad\leq\left(1+\frac{1}{(1-\epsilon)^2}\right)\Ex^P\left[\sum_{s>0}|\Delta M_s|^2\right]\leq \left(1+\frac{1}{(1-\epsilon)^2}\right)\Ex^P[[M,M]_{\infty}^P]<+\infty.
	\end{align*}
	Moreover, it holds that
	\begin{align*}
	\Ex^P\left[\sum_{s>0}|\ln(1+\Delta M_s)|^2\mathds{1}_{\Delta M_s\in(-1,-\epsilon)}\right]&\leq\Ex^P\left[\sum_{s>0}|\Delta M_s|^2\right]+\Ex^P\left[\sum_{s>0}\frac{|\Delta M_s|^2}{(1+\Delta M_s)^2}\mathds{1}_{\Delta M_s\in(-1,-\epsilon)}\right]\nonumber\\
	&\leq\Ex^P[[M,M]_{\infty}^P]+\Ex^P[\Theta_{\infty}^2(\epsilon)]<+\infty.
	\end{align*}
	Using \eqref{eq:ECLambda}, the above estimates imply that $\Ex^P\left[C(W^M)_{\infty}\right]<+\infty$, i.e., $C(W^M)\in{\cal A}^{P,+}$ and hence $W^M\in G_{loc}^P(\mu^Y)$ using Theorem II.1.33-(a) in \cite{JacodShiryaev}. We next consider the case of $M\in{\cal M}_{loc}^P\cap{\cal A}^P$ satisfying $\Theta^1(\epsilon)\in{\cal A}^{P,+}$ for some $\epsilon\in[0,1)$. By~\eqref{eq:1+U+V}, we have that
	\begin{align}\label{eq:ECLambda2}
	\Ex^P\left[\bar{C}(W^M)_{\infty}\right]&= \Ex^P\left[(|W^M|*\nu)_{\infty}\right]=\Ex^P\left[(|W^M|*\mu)_{\infty}\right]=M_{\mu^Y}^P\left(|W^M|\right)\nonumber\\
&\leq M_{\mu^Y}^P\left(|\ln(1+\Delta M)|\right)\\
	&=\Ex^P\left[\sum_{s>0}|\ln(1+\Delta M_s)|\mathds{1}_{\Delta M_s\in[-\epsilon,\infty)}\right]+\Ex^P\left[\sum_{s>0}|\ln(1+\Delta M_s)|\mathds{1}_{\Delta M_s\in(-1,-\epsilon)}\right].\nonumber
	\end{align}
	First, we obtain that
	\begin{align*}
	\Ex^P\left[\sum_{s>0}|\ln(1+\Delta M_s)|\mathds{1}_{\Delta M_s\in[-\epsilon,\infty)}\right]&\leq
	\Ex^P\left[\sum_{s>0}|\Delta M_s|\right]+\Ex^P\left[\sum_{s>0}\frac{|\Delta M_s|}{1+\Delta M_s}\mathds{1}_{\Delta M_s\in[-\epsilon,\infty)}\right]\nonumber\\
	&\leq\frac{2-\epsilon}{1-\epsilon}\Ex^P\left[\sum_{s>0}|\Delta M_s|\right]<+\infty,
	\end{align*}
	and the following inequality also holds
	\begin{align*}
	&\Ex^P\left[\sum_{s>0}|\ln(1+\Delta M_s)|\mathds{1}_{\Delta M_s\in(-1,-\epsilon)}\right]\leq\Ex^P\left[\sum_{s>0}|\Delta M_s|\right]+\Ex^P\left[\Theta_{\infty}^1(\epsilon)\right]
	<+\infty.
	\end{align*}
	The above estimates imply that $\Ex^P\left[\bar{C}(W^M)_{\infty}\right]<+\infty$, i.e., $\bar{C}(W^M)\in{\cal A}^{P,+}$, and hence $W^M\in G_{loc}^P(\mu^Y)$ using Theorem II.1.33-(b) in \cite{JacodShiryaev}.
	
	Recall the expression of $K^M(v)$ given in \eqref{eq:BSDE-sol}. It can be easily seen that $M_{\mu^Y}^P(K^M(v)|\tilde{\cal P})=0$, and it satisfies {Claim} (i), using similar estimates to the ones derived above. Next, recall the definition $Z^M:=H^M$ given in \eqref{eq:BSDE-sol}. Because $H^M\in{\cal P}$, we obtain that $Z^M\in{\cal P}$. Moreover, because $H^M\in L_{loc}^{P,2}(Y^c)$, we deduce that $Z^M\in L_{loc}^{P,2}(Y^c)$. Using \eqref{eq:BSDE-sol}, we conclude that $e^{\Pi^M}={\cal E}(M){\cal E}(V^M)$. Since $V^M\in{\cal V}^P\cap{\cal P}$, we conclude that  $[M,V^M]^P=\Delta M\cdot V^M=\Delta V^M\cdot M=0$ using \eqref{eq:optimcondQn2233} and (3.6) in~\cite{LipShir}, page 119. The Yor's formula yields $e^{\Pi^M}={\cal E}(V^M+M)$. This leads to the equality $V^{\Pi}=e^{\Pi_{-}^M}\cdot V^M$, and hence $e^{-\Pi_{-}^M}\Delta V^{\Pi}=\Delta V^M=0>-1$, $P$-a.s., {by} \eqref{eq:optimcondQn2233}. This verifies Claim (iii). \hfill$\Box$

{The integral representation~\eqref{eq:BSDEspecial} w.r.t. the semimartingale $Y$} is also related to the solution of a forward HJB equation in an integrated semimartingale factor model. We explore such a connection  in Section~\ref{sec:FHJB}.

\section{Forward HJB Equation with Integrated Factor}\label{sec:FHJB}

We incorporate time-monotone performance processes in our semimartingale factor form by allowing the forward performance process to depend on the sample path of the factor process $Y$ via its integral functional. More specifically, we consider a random field of the form:
\begin{align}\label{VbarVtxY}
U_t(x) = G\left(t,x,Y_t,\int_0^t g(s,Y_{s})ds\right),\quad (t,x)\in\R_+^2.
\end{align}
The factor function $G(t,x,y,z):[0,\infty)\times(0,\infty)\times\R^d\times\R\to\R$ belongs to $C^{1,2,2,1}$ and $g:[0,\infty)\times\R^d\to\R$ is a Borel function. If the input function $G(t,x,y,z)=U_0(x)\Gamma(t,y,z)$ with $U_0(x)=\frac{1}{p}x^pe^{D_0}$ ($D_0$ is an input which determines the initial forward preference), then $\Gamma(t,y,z)$ is the solution to a forward (partial integro-differential) HJB equation (see \eqref{eq:forhjbgamma} below).

Unlike the classical HJB equation, the forward HJB equation \eqref{eq:forhjbgamma} specifies the initial rather than the terminal value, and needs to be solved forward in time. Therefore, it is an ill-posed (time-reversed) HJB equation, and the classical theory of HJB equations is not applicable. We bypass this difficulty by connecting the solution of the forward HJB equation~\eqref{eq:forhjbgamma} to {the triplet of processes $(\Pi^M,Z^M,W^M)$ analyzed in the previous section. We prove that solving the forward equation is equivalent to constructing} {the process $\Pi^M$ with representation \eqref{eq:BSDEspecial}} and satisfying the factor form \eqref{VbarVtxY}. We establish a closed-form factor representation of {$(\Pi^M,Z^M,W^M)$} when the factor $Y$ is a special semimartingale (see Lemma~\ref{lem:PiMY22} below).  

Throughout the section, we consider the predictable characteristics of $R$ and $Y$ admitting the factor representation:
\begin{align}\label{eq:MRY0}
(B(\omega),C(\omega),\nu(\omega,dt,du))&=\left(\int_0^{\cdot}b(s,Y_{s}(\omega))ds,\int_0^{\cdot}c(s,Y_{s}(\omega))ds,F(t,Y_{t}(\omega),du)dt\right)\\
(B^Y(\omega),C^Y(\omega),\nu^Y(\omega,dt,dv))&=\left(\int_0^{\cdot}b^Y(s,Y_{s}(\omega))ds,\int_0^{\cdot}c^{Y}(s,Y_{s}(\omega))ds,F^Y(t,Y_{t}(\omega),dv)dt\right).\nonumber
\end{align}
Here $b(t,y)$ ($b^Y(t,y)$), $c(t,y)$ ($c^Y(t,y)$), and $F(t,y,du)$ ($F^Y(t,y,dv)$) satisfy conditions (i) and (ii) of Theorem III.2.32 in \cite{JacodShiryaev}. We also set the portfolio constraint to be in a factor form as
${\cal C}_t(\omega):=C_t(Y_{t-}(\omega))\subseteq\R^n$,
where $C:[0,\infty)\times\R^m\to{\cal B}(\R^n)$.

\subsection{Forward HJB equations and connection to factor form of BSDEs}\label{sec:formsolv}

Recall the wealth process given in Section~\ref{sec:marketmodel}. It follows from Definition~\ref{def:FIPP} and \eqref{VbarVtxY} that, if $U_t(x)=U_0(x)\Gamma(t,Y_t,\int_0^t g(Y_{s})ds)$ is a $P$-FIPP, then the pair of functions $(\Gamma,g)$ satisfies the following equation: for $(t,y,z)\in[0,\infty)\times\R^d\times\R$,
\begin{align}\label{eq:forhjbgamma}
0&=\partial_t\Gamma(t,y,z)+\partial_z\Gamma(t,y,z)g(t,y)+\nabla_y\Gamma(t,y,z)^{\top}b^{Y}(t,y)+\frac{1}{2}{\rm tr}\left[\nabla_{yy}^2\Gamma(t,y,z)c^{Y}(t,y)\right]\nonumber\\
&\qquad+p\sup_{\pi\in{\cal C}_{0,t}(y)\cap{C}_t(y)}\bigg\{\Gamma(t,y,z)\pi^{\top}b(t,y)+\frac{p-1}{2}\Gamma(t,y,z)\pi^{\top}c(t,y)\pi\nonumber\\
&\qquad\qquad+\pi^{\top}c^{RY}(t,y)\nabla_y\Gamma(t,y,z)+\int\big\{p^{-1}(1+\pi^{\top}u)^p\Gamma(t,y+v,z)-p^{-1}\Gamma(t,y,z)\nonumber\\
&\qquad\qquad-\Gamma(t,y,z)\pi^{\top}h(u)-p^{-1}\nabla_y\Gamma(t,y,z)^{\top}h^Y(v)\big\}\overline{F}(t,y,du,dv)\bigg\},
\end{align}
with initial value $\Gamma(0,Y_0,0)=1$. Here $\int_0^{\cdot}c_s^{RY}ds=(\lc R^{c,i},Y^{c,j}\rc)_{i=1,\ldots,n}^{j=1,\ldots,d}$, and we have used the notations: $\partial_t:=\frac{\partial}{\partial t}$, $\partial_z:=\frac{\partial}{\partial z}$, $\nabla_y:=(\partial_{y_1},\ldots,\partial_{y_d})^{\top}$ and $\nabla_{yy}^2:=(\partial_{y_iy_j}^2)_{i,j=1,\ldots,d}$.

Next, we connect the solution of the forward equation \eqref{eq:forhjbgamma} to the factor representation of {$\Pi^M$ specified by~\eqref{VbarVtxY}}. To this purpose, let $M\in\bar{\cal M}_{loc}^{P,+}$ be such that the Jacod's representation of $M$ w.r.t. $R$ and $Y$ is respectively given by
\begin{align*}
(H(t,Y_{t-}(\omega)),\Xi(t,Y_{t-}(\omega),u),0,0),~\text{and}~(\Lambda(t,Y_{t-}(\omega)),e^{\theta(t,Y_{t-}(\omega),v)}-1,0,0).
\end{align*}
The deterministic functions $H(t,y)$, $\Xi(t,y,u)$, $\Lambda(t,y)$ and $\theta(t,y,v)$ are all Borel measurable. Then, $P$-a.s.
\begin{align}\label{eq:M12}
M=H\cdot R^c + \Xi(u)*(\mu-\nu)=\Lambda\cdot Y^c + \{e^{\theta(v)}-1\}*(\mu^Y-\nu^Y).
\end{align}
The decomposition~\eqref{eq:M12} implies the existence of a relation between Jacod's representations of $M$ w.r.t. $R$ and w.r.t. $Y$. We state this relation in the following lemma.
\vspace{-0.3cm}
\begin{lemma}\label{lem:HtildeHc}
We have ${H}^{\top}c^{RY}\Lambda=\Lambda^{\top}c^{YR}{H}=H^{\top}cH=\Lambda^{\top}c^{Y}\Lambda$, and $\Xi(\Delta R) = e^{\theta(\Delta Y)}-1$, \ASP
\end{lemma}
\vspace{-0.2cm}

The following main result connects the solution of the forward equation \eqref{eq:forhjbgamma} satisfied by $(\Gamma,g)$ to the factor form of {$\Pi^M$ specified by \eqref{VbarVtxY}}. The proof is reported in the Appendix.

\begin{theorem}\label{thm:solhjb}
Let $M\in\bar{\cal M}_{loc}^{P,+}$ be given as in \eqref{eq:M12}. If the component $\Pi^M$ {in the integral representation \eqref{eq:BSDEspecial}} with initial value $\Pi_0^{M}=0$ admits a factor representation (i.e., there exist Borel functions $\Pi\in C^{1,2,1}$ and $\tilde{g}$ such that $\Pi_t^M=\Pi(t,Y_t,\int_0^t\tilde{g}(s,Y_{s})ds)$
	for $t\geq0$), then the forward HJB equation \eqref{eq:forhjbgamma} admits a classical solution $(\Gamma(t,y,z),g(t,y))=(e^{\Pi(t,y,z)},\tilde{g}(t,y))$
	on $(t,y,z)\in[0,\infty)\times\R^d\times\R$.
\end{theorem}

\subsection{Factor representation of {$(\Pi^M,Z^M,W^M)$}}\label{sec:factrep}

In this section, we establish the factor representation \eqref{VbarVtxY} for the triplet of {processes $(\Pi^M,Z^M,W^M)$ admitting the integral representation \eqref{eq:BSDEspecial}} when the function $g$ is nonzero. We provide a closed-form expression for $g$ when $Y$ is a special semimartingale factor.

The following lemma provides an {explicit form for the integral representation \eqref{eq:BSDEspecial} w.r.t. the semimartingale factor}:
\begin{lemma}\label{lem:PiMY}
Let $M\in\bar{\cal M}_{loc}^{P,+}$ admit the Jacod's decomposition in \eqref{eq:M12}. Then, {the triplet of processs $(\Pi^M,Z^M,W^M)$ defined in~\eqref{eq:BSDE-sol} admits the representation}
\begin{align}\label{eq:solutionspeci0}
\Pi_t^M &=\int_0^t\Psi(s,Y_{s})ds + \int_0^t\Lambda(s,Y_{s-})d Y_s^c+\int_0^t\int\theta(Y_{s-},v)(\mu^Y-\nu^Y)(ds,dv),\nonumber\\
Z_t^M &=\Lambda(t,Y_{t-}),\quad W_t^M(v)=\theta(t,Y_{t-},v).
\end{align}
For $(t,y)\in[0,\infty)\times\R^d$, we have defined
\begin{align}\label{eq:varphity}
\Psi(t,y)&:=-p\varphi(t,y)-\frac{1}{2}\Lambda(t,y)^{\top}c^{Y}(t,y)\Lambda(t,y)-\int\{e^{\theta(t,y,v)}-1-\theta(t,y,v)\}F^Y(t,y,dv),
\end{align}
where the function $\varphi(t,y)$ is given by
\begin{align}\label{eq:PhiMspecial}
\varphi(t,y)&:=\sup_{\pi\in{C}_t(y)\cap{\cal C}_{0,t}(y)}\Bigg\{\pi^{\top}\left(b(t,y) +c(t,y)H(t,y)\right)+\frac{p-1}{2}\pi^{\top} c(t,y)\pi\nonumber\\
&\quad+\int\{p^{-1}(1+\pi^{\top}u)^p-p^{-1}-\pi^{\top}h(u)\}F(t,y,du)\\
&\quad+\int\{p^{-1}(1+\pi^{\top}u)^p-p^{-1}\}\{e^{\theta(t,y,v)}-1\}\overline{F}(t,y,du,dv)\Bigg\}.\nonumber
\end{align}
\end{lemma}

If the factor process $Y$ is a special semimartingale, the process $\Pi^M$ in~\eqref{eq:BSDE-sol} admits an integrated semimartingale factor representation:
\begin{lemma}\label{lem:PiMY22}
	Let $\sigma\in\R^d$ and consider the local martingale $M$ in \eqref{eq:M12} with $\Lambda\equiv\sigma$ and $\theta(v)=\sigma^{\top}v$. If $Y$ is special, then {the process $(\Pi^M,Z^M,W^M)$ given in~\eqref{eq:BSDE-sol}} admits the representation:
	\begin{align}\label{eq:solutionspeci022}
	\Pi_t^M &=\int_0^t\Psi^{\sigma}(s,Y_{s})ds + \sigma^{\top}(Y_t-Y_0),\quad
	Z^M =\sigma,\quad W^M(v)=\sigma^{\top}v.
	\end{align}
	For $(t,y)\in[0,\infty)\times\R^d$, we define
	\begin{align}\label{eq:varphity22}
	\Psi^{\sigma}(t,y)&:=-p\varphi(t,y)-\frac{1}{2}\sigma^{\top}c^{Y}(t,y)\sigma-\sigma^{\top}b^Y(t,y)-\int\{e^{\sigma^{\top}v}-1-\sigma^{\top}v\}F^Y(t,y,dv).
	\end{align}
	The function $\varphi(t,y)$ is given by \eqref{eq:PhiMspecial}. In addition, the function $\theta(v)=\sigma^{\top}v$, and $H$ satisfies the relation given in Lemma~\ref{lem:HtildeHc} with $\Lambda=\sigma$.
\end{lemma}
\begin{remark}\label{rem:solvspeicalfactor}
Lemma~\ref{lem:PiMY} suggests the existence of a solution of the form  $(\Pi(t,y,z),g(t,y))=(\tilde{\Pi}(t,y)+z,\Psi(t,y))$, where the function $\Psi$ is given by \eqref{eq:varphity}. Using this solution structure, Lemma~A.3 in the Appendix may be restated as follows: under the conditions of Lemma~\ref{lem:PiMY}, the component $\Pi^M$ of~\eqref{eq:BSDEspecial}, {with the additional constraint $\Pi_0^M=0$}, admits a factor representation $(\Pi,\Psi)$ if and only if $\Pi(t,y,z)=\tilde{\Pi}(t,y)+z$ and $\tilde{\Pi}$ is a classical solution of the forward equation: for $(t,y)\in[0,\infty)\times\R^d$,
	\begin{align}\label{eq:tildePiPDEPig}
	-\Psi(t,y)&=\partial_t\tilde{\Pi}(t,y)+{\cal A}^Y\tilde{\Pi}(t,y)+f(t,y,\nabla_y\tilde{\Pi}(t,y),\tilde{\Pi}(t,y+v)-\tilde{\Pi}(t,y))
	\end{align}
	with the initial condition $\tilde{\Pi}(0,Y_0)=0$. Moreover, Lemma \ref{lem:PiMY22} indicates the existence of an explicit solution of the forward equation \eqref{eq:PDEPig} when the local martingale $M$ is given by~\eqref{eq:M12}, and the factor process $Y$ is a special semimartingale. The explicit solution of~\eqref{eq:PDEPig} is given by
	\begin{align}\label{eq:sol1}
	(\Pi(t,y,z),g(t,y))=(\sigma^{\top}(y-Y_0)+z,\Psi^{\sigma}(t,y)),
	\end{align}
	where $\Psi^{\sigma}$ is given by \eqref{eq:varphity22}. It may be easily verified that $\tilde{\Pi}(t,y)=\sigma^{\top}(y-Y_0)$ is a classical solution of \eqref{eq:tildePiPDEPig} if $\Psi$ is replaced by $\Psi^{\sigma}$.
\end{remark}

Building on Remark~\ref{rem:solvspeicalfactor}, we revisit a Black-Scholes market with It\^o diffusion factors of the form \eqref{VbarVtxY} and nonzero $g$. 
\begin{example}\label{exam:familypoweFIPP}
In a one-dimensional continuous diffusion model, \cite{NadZar14} characterize the solution to the forward Cauchy problem when the coefficients in the It\^o representation of $Y$ are sufficiently smooth. \cite{LiangZari17} consider a multi-dimensional continuous diffusion model and account for portfolio constraints. They assume that the return-factor process pair $(R,Y)$ follows the dynamics $dR_t=b^R(Y_t)dt+\sigma^R(Y_t)dB_t$ and $dY_t=b^Y(Y_t)dt+\sigma^YdB_t$, where $B$ is a $d$-dimensional Brownian motion, and the factor process $Y$ is also $d$-dimensional. They additionally require the covariance matrix $\sigma^Y(\sigma^Y)^{\top}$ to be positive definite, and hence invertible. They prove that the solution component $\Pi$ of their infinite horizon BSDE admits a factor representation of the form $\Pi_t=G(Y_t)$, for some function $G:\R^d\to\R$ that exhibits at most linear growth. We next analyze the model by \cite{LiangZari17} under the factor form \eqref{VbarVtxY} but with a nonzero $g$. We show that this yields a family of $P$-FIPPs. Let $\tilde{\sigma}$ be an arbitrary $\R^d$-valued column vector. Take the scalar local martingale $M=\tilde{\sigma}^{\top}B$. Then $M$ is a continuous $P$-martingale and ${\cal E}(M)$ is a $P$-martingale. Hence, $d\Qx^{\tilde{\sigma}}={\cal E}(M)dP$ defines a probability measure $\Qx^{\tilde{\sigma}}$, and  $B_t^{\tilde{\sigma}}:=B_t-\tilde{\sigma}t$, $t\geq0$, is a $d$-dimensional Brownian motion under $\Qx^{\tilde{\sigma}}$. In the It\^o diffusion model discussed above, \eqref{eq:BSDEspecial} reduces to
	\begin{align}\label{eq:BSDEspecialBM}
	d\Pi^{\tilde{\sigma}} = -{f}({Z}^{\tilde{\sigma}})dt + Z^{\tilde{\sigma}}\cdot dY^c,\qquad Y^c=\sigma^YB,
	\end{align}
	where the driver $f$ is given by
	$f(Z) = \frac{1}{2}|Z^{\top}\sigma^Y|^2 + p\sup_{\pi\in{\cal C}}\Phi_p(\pi;Y,Z)$,
	and
	\begin{align}\label{eq:PhiMpiBM}
	\Phi_p(\pi;y,Z) &= \pi^{\top}\{b^R(y) + \sigma^R(y)(\sigma^Y)^{\top}Z\}+\frac{p-1}{2}\pi^{\top}\sigma^R(y)\sigma^R(y)^{\top}\pi.
	\end{align}
	Since $Y^c=\sigma^YB$ by \eqref{eq:BSDEspecialBM}, we obtain the martingale representation \eqref{eq:M12} with $\tilde{H}=\sigma^Y[(\sigma^Y)^{\top}\sigma^Y]^{-1}\tilde{\sigma}$. Moreover, the factor $Y$ is an It\^o diffusion process, and hence it is a QLC special semimartingale. It follows from Lemma~\ref{lem:PiMY22} that $\Pi_t^{\tilde{\sigma}}=\Pi^{\tilde{\sigma}}(t,Y_t,\int_0^tg^{\tilde{\sigma}}(s,Y_s)ds)$ for $t\geq0$, and $\Pi_0^{\tilde{\sigma}}=0$ admits a factor representation:
	\begin{align}\label{eq:factorBM}
	\Pi^{\tilde{\sigma}}(t,y,z)&:=\tilde{H}^{\top}(y-Y_0)+z,\\
	g^{\tilde{\sigma}}(t,y)&:=-p\sup_{\pi\in{\cal C}}\left\{\pi^{\top}(b^R(y) + \sigma^R(y)(\sigma^Y)^{\top}\tilde{H})+\frac{p-1}{2}\pi^{\top}\sigma^R(y)\sigma^R(y)^{\top}\pi\right\}
	-\frac{1}{2}|\tilde{H}^{\top}\sigma^Y|^2-\tilde{H}^{\top}b^Y(y).\nonumber
	\end{align}
	Assume $\sigma^R(y)$ has full rank. Then $\sigma^R(y)^{\top}{\cal C}$ is also closed. The market price of risk is then given by $\lambda(y):=\sigma^R(y)^{\top}[\sigma^R(y)\sigma^R(y)^{\top}]^{-1}b^R(y)$, i.e.,  $\sigma^R(y)\lambda(y)=b^R(y)$. The optimal strategy $\pi^{\tilde{\sigma},*}$ thus satisfies that
	\begin{align}\label{eq:optimal-strategy}
	\sigma^R(y)^{\top}\pi^{\tilde{\sigma},*}\in{\sf P}_{\sigma^R(y)^{\top}{\cal C}}\left\{(1-p)^{-1}(\lambda(y)+\tilde{\sigma})\right\}.
	\end{align}
	Here, for any nonempty closed subset $K$ of $\R^n$, ${\sf P}_K\{ x \}$ is defined as the projection that maps a vector $x\in\R^n$ to the points in $K$ with minimal distance from $x$. Define $G^{\tilde{\sigma}}(t,x,y,z):=U_0(x)e^{\Pi^{\tilde{\sigma}}(t,y,z)}$. By Lemma~\ref{lem:BSDE-utility}, if $Q^{\tilde{\sigma}}:={\cal E}(p\sigma^R(Y)^{\top}\pi^{\tilde{\sigma},*}\cdot B^{\tilde{\sigma}})$ is a $\Qx^{\tilde{\sigma}}$-martingale, then
	\begin{align}\label{eq:Utildesigma}
	U_t^{\tilde{\sigma}}(x):=U_0(x)e^{\Pi_t^{\tilde{\sigma}}}=G^{\tilde{\sigma}}\left(t,x,Y_{t},\int_0^tg^{\tilde{\sigma}}(s,Y_s)ds\right),\quad t\geq0
	\end{align}
	is a $P$-FIPP. The above representation of $U^{\tilde{\sigma}}$ in terms of $G^{\tilde{\sigma}}$ gives a solution of the forward HJB equation \eqref{eq:forhjbgamma}; see Theorem~\ref{thm:solhjb}. Moreover, it can also be seen that the FIPP $U^{\tilde{\sigma}}$ depends on the vector $\tilde{\sigma}$. Hence, we have established a family of power FIPPs given by $(U^{\tilde{\sigma}})_{\tilde{\sigma}\in\R^d}$, that results in a family of optimal strategies given by $(\pi^{\tilde{\sigma},*})_{\tilde{\sigma}\in\R^d}$ if $(Q^{\tilde{\sigma}})_{\tilde{\sigma}\in\R^d}$ is a family of $\Qx^{\tilde{\sigma}}$-martingales.
\end{example}

\vspace{0.3cm}
\begin{APPENDICES}
\section*{Appendix: Proofs of Auxiliary Results}\label{app:proof1}
\renewcommand\theequation{A.\arabic{equation}}
\setcounter{equation}{0}

This section collects the technical proofs of some important auxiliary results that are used to establish theorems, propositions, and lemmas in the main body of the paper.

\noindent{\bf Proof of Lemma~\ref{lem:ReptildeMA}.}\quad We provide the proof for a general increasing and predictable process {$A=(A_t)_{t\geq0}$ used in the representation of predictable characteristics for semimartingales}. Then, the desired result follows by taking $A_t=t$.
An application of It\^o formula yields $L=L_0{\cal E}(N)$, where the process
		$N
		=B^D +D^c+ \frac{C^{D}}{2}+ h^D(v)*(\mu^D-\nu^D) + (e^v-1-h^D(v))*\mu^D$.
		Note that $D$ is exponentially special by the assumption {\Ad}-(i) or (ii). Then, we have that
		\begin{align}\label{eq:AML}
		M^L&=L_0 {\cal E}(N)_{-}\cdot\left\{D^c+ (e^v-1)*(\mu^D-\nu^D)\right\},\nonumber\\
		V^L&=L_0{\cal E}(N)_{-}\cdot\left\{B^D+ \frac{C^{D}}{2}+(e^v-1-h^D(v))*\nu^D\right\}.
		\end{align}
		Note that $\Delta B^D=\int h^D(v)\nu^D(\{\cdot\}, dv)$, so it follows from \eqref{eq:AML} that
		\begin{align}\label{eq:deltaAL}
		\Delta V^L 
		&=L_0{\cal E}(N)_{-}\int\{e^v-1\}\nu^D(\{\cdot\}, dv)=L_0{\cal E}(N)_{-}\left(\int e^v\nu^D(\{\cdot\}, dv)-a^D\right),
		\end{align}
        where $a^D:=\int \nu^D(\{\cdot\},dv)$. {By} Theorem~2.5.1 in~\cite{LipShir}, page 127, we have
		\begin{align*}
		M &= (L_{-}+\Delta V^L)^{-1}\cdot M^L=\left\{\int e^v\nu^D(\{\cdot\},dv)+1-a^D\right\}^{-1}\cdot\left\{D^c+ (e^v-1)*(\mu^D-\nu^D)\right\},
		\end{align*}
		and hence $V= L_{-}^{-1}\cdot V^L = B^D+ \frac{C^{D}}{2}+\{e^v-1-h^D(v)\}*\nu^D$. Using the expression for $M$ derived above, we obtain that
		\begin{align*}
		\Delta M &=\left\{\int e^v\nu^D(\{\cdot\}, dv)+1-a^D\right\}^{-1}(e^v-1)*(\mu^D-\nu^D)(\{\cdot\}, dv)\nonumber\\
		&=\frac{(e^{\Delta D}-1){\I}_{\Delta D\neq0}+a^D-\int e^v\nu^D(\{\cdot\},dv)}{\int e^v\nu^D(\{\cdot\}, dv)+1-a^D}.
		\end{align*}
		Using \eqref{eq:deltaAL}, we get $\Delta V=\int e^v\nu^D(\{\cdot\}, dv)-a^D$. Note that, it follows from Proposition II.2.9 in \cite{JacodShiryaev} that $a_t^D=\Delta A_tF_t^D(\R_0^n)$. This yields that $a^D\equiv0$ when $A_t=t$ for $t\geq0$ and hence $\int e^v\nu^D(\{\cdot\}, dv)\equiv0$ when $A_t=t$ for $t\geq0$. From these, we conclude the lemma. \hfill$\Box$

\noindent{\bf Proof of Lemma~\ref{lem:equifipp}.}\quad To prove Lemma~\ref{lem:equifipp}, we first need an auxiliary result. For any $T\in(0,\infty)$, let ${\cal T}_{T}$ be the set of $\Fx$-stopping times $\tau\leq T$. 
Then, we have the following equivalent characterization:

\noindent{\bf Lemma~A.1}\quad \emph{An adapted process $\zeta=(\zeta_t)_{t\geq0}$ is of class of {\rm(DL)} under a probability measure $Q$ if and only if for any $T\in(0,\infty)$,
			\begin{itemize}
				\item[{\rm(i)}] $\sup_{\tau\in{\cal T}_{T}}\Ex^{Q}[|\zeta_{\tau}|]<+\infty$;
				\item[{\rm(ii)}] for any $\varepsilon>0$, there is $\delta>0$ s.t. whenever $Q(A)\leq\delta$ with $A\in\F_T$, we have
				\[
				\sup_{\{\tau\in{\cal T}_T;\ A\in\F_{\tau}\}}\Ex^{Q}[|\zeta_{\tau}|\I_{A}]\leq \varepsilon.
				\]
			\end{itemize}}

\noindent{\bf Proof.}\quad We first assume that $\zeta=(\zeta_t)_{t\geq0}$ is of class of (DL) under $Q$. Then (i) follows immediately.  On the other hand, for any $\varepsilon>0$, there exists $\delta>0$ s.t. whenever $Q(A)\leq\delta$ with $A\in\F_T$, we have that
			\begin{align*}
			\Ex^{Q}[|\zeta_{\tau}|\I_{A}]=\Ex^{Q}[|\zeta_{\tau}|\I_{A\cap\{|\zeta_{\tau}|>\lambda\}}]+\Ex^{Q}[|\zeta_{\tau}|\I_{A\cap\{|\zeta_{\tau}|\leq\lambda\}}],\quad\forall~\lambda>0.
			\end{align*}
			Since $(\zeta_{\tau})_{\tau\in{\cal T}_T}$ is U.I., there exists a sufficiently large $\lambda>0$ s.t. $\sup_{\tau\in{\cal T}_{T}}\Ex^{Q}[|\zeta_{\tau}|\I_{|\zeta_{\tau}|>\lambda}]\leq\varepsilon/2$. Take $\delta=\varepsilon/(2\lambda)$. It then holds that $\sup_{\tau\in{\cal T}_T}\Ex^{Q}[|\zeta_{\tau}|\I_{A}]=\varepsilon/2+\lambda Q(A)\leq \varepsilon$. This yields the inequality $\sup_{\{\tau\in{\cal T}_T;\ A\in\F_{\tau}\}}\Ex^{Q}[|\zeta_{\tau}|\I_{A}]\leq\sup_{\tau\in{\cal T}_T}\Ex^{Q}[|\zeta_{\tau}|\I_{A}]\leq\varepsilon$, i.e., (ii) holds.
			
We next assume that (i) and (ii) hold. Let $\alpha:=\sup_{\tau\in{\cal T}_{T}}\Ex^{Q}[|\zeta_{\tau}|]$, and hence $\alpha\in[0,\infty)$ by (i). Then there exists $\delta>0$ s.t. for all $\hat{\tau}\in{\cal T}_T$, $Q(A_{\hat{\tau}}^{\lambda})\leq\delta$ for $\lambda\geq\alpha/\delta$. Here, we have set $A_{\hat{\tau}}^{\lambda}:=\{|\zeta_{\hat{\tau}}|>\lambda\}$. In particular, $\hat{\tau}\in\{\tau\in{\cal T}_T;\ A_{\hat{\tau}}^{\lambda}\in\F_{\tau}\}$. Therefore, for sufficiently large $\lambda\geq\alpha/\delta$, it holds from (ii) that
\begin{align}\label{eq:esitNI}
			\Ex^{Q}\left[|\zeta_{\hat{\tau}}|\I_{A_{\hat{\tau}}^{\lambda}}\right]\leq \sup_{\{\tau\in{\cal T}_T;\ A_{\hat{\tau}}^{\lambda}\in\F_{\tau}\}}\Ex^{Q}\left[|\zeta_{\tau}|\I_{A_{\hat{\tau}}^{\lambda}}\right]\leq\varepsilon.
\end{align}
Since $\hat{\tau}\in{\cal T}_T$ and $\varepsilon>0$ are arbitrary,  it follows from~\eqref{eq:esitNI} that $(\zeta_{\tau})_{\tau\in{\cal T}_T}$ is U.I. under $Q$. This ends the proof. \hfill$\Box$

We are now in a position to prove Lemma~\ref{lem:equifipp}. For (i), by Definition~\ref{def:FIPP}, it is enough to verify that $\pi \in \Gamma_{\cal C}^{Q^M}$ if and only if $\pi \in \Gamma_{\cal C}^{P}$. By Definition~\ref{def:control-set}, it is sufficient to show that $U_{t}(x{\cal E}_{t}(\pi\cdot R))^{-}$, $t\geq0$, is of class (DL) under $P$ if and only if $U_{t}^V(x{\cal E}_{t}(\pi\cdot R))^{-}$, $t\geq0$, is of class (DL) under $Q^M$. For notational convenience, set $\zeta_t:=U_{t}^V(x{\cal E}_{t}(\pi\cdot R))^{-}$ and hence $U_{t}(x{\cal E}_{t}(\pi\cdot R))^{-}={\cal E}_t(M)\zeta_t$ because ${\cal E}_t(M)>0$.  We first assume that ${\cal E}_t(M)\zeta_t$, $t\geq0$, is of class (DL) under $P$. Then, it follows from Lemma~{A.1} that (i) $\sup_{\tau\in{\cal T}_{T}}\Ex^{P}[{\cal E}_{\tau}(M)\zeta_{\tau}]<+\infty$, and (ii) for any $\varepsilon>0$, there exists $\delta>0$ s.t. whenever $P(A)\leq\delta$ with $A\in\F_T$, it holds that $\sup_{\{\tau\in{\cal T}_T;\ A\in\F_{\tau}\}}\Ex^{P}[{\cal E}_{\tau}(M)\zeta_{\tau}\I_{A}]\leq \varepsilon$. By \eqref{eq:QM}, we obtain $\sup_{\tau\in{\cal T}_{T}}\Ex^{Q^M}[\zeta_{\tau}]<+\infty$. Since $Q^M\sim P$ and $P(A)\leq\delta$ for $A\in\F_T$, there exists $\hat{\delta}>0$ s.t. $Q^M(A)=\Ex^P[{\cal E}_T(M)\I_{A}]\leq\hat{\delta}$. By (ii) and \eqref{eq:QM}, we also have that $\sup_{\{\tau\in{\cal T}_T;\ A\in\F_{\tau}\}}\Ex^{Q^M}[\zeta_{\tau}\I_{A}]\leq \varepsilon$. This shows that $\zeta=(\zeta_t)_{t\geq0}$ is of class (DL) under $Q^M$ using again Lemma~{A.1}. Similarly, we can verify that if $\zeta_t$, $t\geq0$, is of class (DL) under $Q^M$, then ${\cal E}_t(M)\zeta_t$, $t\geq0$, is of class (DL) under $P$.

{We next prove (ii). Let us assume that the random field $U^V$ defined by \eqref{eq:tildeU} is a $Q^M$-FIPP with optimal trading strategy $\pi^*\in\Gamma_{\cal C}^{Q^M}$. It can be easily seen that $x\to U_t(x):=U_0(x){\cal E}_t(M){\cal E}_t(V)$ is concave and increasing because $U_0(x)=\frac{e^{D_0}}{p}x^p$ for $x>0$. It follows from Definition~\ref{def:FIPP} with $Q=Q^M$ that, for any $\pi\in\Gamma_{\cal C}^{Q^M}$, $U^V(X^{\pi,x})=U_0(X^{\pi,x}){\cal E}(V)$ is a $Q^M$-supermartingale. This yields $\Ex^{Q^M}[U_t^V(X_t^{\pi,x})|\F_s]\leq U_s^V(X_s^{\pi,x})$ for any $0\leq s<t<\infty$. In view of \eqref{eq:QM} and using Eq.~(3.9) from Chapter III of \cite{JacodShiryaev}, page 168, it follows that $\Ex^{Q^M}[U_t^V(X_t^{\pi,x})|\F_s]=\Ex^{P}[U_t(X_t^{\pi,x})|\F_s]/{\cal E}_s(M)$ because $M\in{\cal M}_{loc}^{P,+}$. Hence, $\Ex^{P}[U_t(X_t^{\pi,x})|\F_s]\leq {\cal E}_s(M)U_s^V(X_s^{\pi,x})=U_s(X_s^{\pi,x})$, $P$-a.s. by using \eqref{eq:UtxL2}. Together with (i), this implies that, for any $\pi\in\Gamma_{\cal C}^{P}$, $U(X^{\pi,x})$ is a $P$-supermartingale. Moreover, it follows from (i) that $\pi^*$ is also admissible under $P$, i.e., $\pi^*\in\Gamma_{\cal C}^{P}$, and $U^V(X^{\pi^*,x})$ is a $Q^M$-martingale. Therefore, $\Ex^{Q^M}[U_t^V(X_t^{\pi^*,x})|\F_s]=U_s^V(X_s^{\pi^*,x})$ for all $0\leq s<t<\infty$. Moreover, by \eqref{eq:UtxL2} and \eqref{eq:QM}, we have that $U(X^{\pi^*,x})$ is a $P$-martingale. {Hence}, we have proven that $U$ defined by \eqref{eq:UtxL2} is a $P$-FIPP with optimal trading strategy $\pi^*\in\Gamma_{\cal C}^P$.
	
In the sequel, assume that the random field $U$ defined by \eqref{eq:UtxL2} is a $P$-FIPP with  optimal trading strategy $\pi^*\in\Gamma_{\cal C}^P$. It follows from Definition~\ref{def:FIPP} with $Q=P$ that, for any $\pi\in\Gamma_{\cal C}^{Q^M}$, $U(X^{\pi,x})={\cal E}_t(M)U^V(X^{\pi,x})$ is a $P$-supermartingale. This yields the inequality $\Ex^P[{\cal E}_t(M)U_t^V(X_t^{\pi,x})|\F_s]\leq{\cal E}_s(M)U_s^V(X_s^{\pi,x})$ for all $0\leq s<t<\infty$. Using \eqref{eq:QM}, we obtain $\Ex^P[{\cal E}_t(M)U_t^V(X_t^{\pi,x})|\F_s]/{\cal E}_s(M)=\Ex^{Q^M}[U_t^V(X_t^{\pi,x})|\F_s]$. Together with (i), this implies that $U^V(X^{\pi,x})$ is a $Q^M$-supermartingale for all $\pi\in\Gamma_{\cal C}^{Q^M}$. Similarly, we can verify from (i) that $\pi^*\in\Gamma_{\cal C}^{Q^M}$, and $U^V(X^{\pi^*,x})$ is a $Q^M$-martingale. This shows that the random field $U^V$ defined by \eqref{eq:tildeU} is a $Q^M$-FIPP with  optimal trading strategy $\pi^*\in\Gamma_{\cal C}^{Q^M}$. This concludes the proof of Lemma~\ref{lem:equifipp}. \hfill$\Box$
}

\noindent{\bf Proof of Lemma~\ref{lem:sdehatA}.}\quad  We provide the proof for a general increasing and predictable process {$A=(A_t)_{t\geq0}$ used in the representation of predictable characteristics for semimartingales}. The desired result then follows by taking $A_t=t$. Observe that, for any $\pi\in{\Gamma}_{\cal C}^{Q^M}$, it holds that
		$U_t^V(X_t^{\pi,x})={U}_0(x){\cal E}_t(V){\cal E}_t(\pi\cdot R)^{p}={U}_0(x){\cal E}_t(V){\cal E}_t(N(\pi))$,
		where the stochastic logarithm $N(\pi)$ is given by
		\begin{align}\label{eq:tildeN1}
		N(\pi) &:= p\pi\cdot R + \frac{p(p-1)}{2}\pi^{\top}C\pi + \{(1+\pi^{\top}u)^p-1-p\pi^{\top}u\}*\mu.
		\end{align}
		Observe that $V\in{\cal V}^{Q^M}\cap{\cal P}$. Then, using (3.6) in \cite{LipShir}, page 119, we have
		\begin{align}\label{eq:quadAN}
		[V,N(\pi)]^{Q^M}=\Delta N(\pi)\cdot V=\Delta V\cdot N(\pi).
		\end{align}
		{From the} Yor's formula, it follows that
		\begin{align}\label{eq:hatUrep}
		U^V(X^{\pi,x})={U}_0(x){\cal E}(V+N(\pi)+[V,N(\pi)]^{Q^M})={U}_0(x){\cal E}(V+(1+\Delta V)\cdot N(\pi)).
		\end{align}
		{By} Definition~\ref{def:FIPP}, for any $\pi\in\Gamma_{\cal C}^{Q^M}$ and $x\in\R_+$, $U^V(X^{\pi,x})$ is a $Q^M$-supermartingale, and there exists ${\pi}^*\in\Gamma_{\cal C}^{Q^M}$ such that $U^V(X^{{\pi}^*,x})$ is a $Q^M$-martingale. By~\eqref{eq:hatUrep},
		for any $\pi\in\Gamma_{\cal C}^{Q^M}$, ${p}^{-1}\{V+(1+\Delta V)\cdot N(\pi)\}$ is a local $Q^M$-supermartingale
		and ${p}^{-1}\{V+(1+\Delta V)\cdot N({\pi}^*)\}$ is a local $Q^M$-martingale. As $Q^M\sim P$, by Lemma~\ref{lem:ReptildeMA}, $Q^M(d\omega)\otimes dA_t$-a.s.
		\begin{align}\label{eq:1addDeltaA}
		1+\Delta V &= 1-a^D+\int e^v\nu^D(\{\cdot\},dv).
		\end{align}
For $\pi\in\Gamma_{\cal C}^{Q^M}$ and $(t,u,\omega)\in[0,\infty)\times\R^n\times\Omega$, define
\begin{align}\label{eq:HWpi2}
H_t^{\pi}(\omega) &:=\left(1-a_t^D(\omega)+\int e^v\nu^D(\omega,\{t\},dv)\right)\pi_t(\omega);\nonumber\\
W_t^{\pi}(\omega,u) &:=\left(1-a_t^D(\omega)+\int e^v\nu^D(\omega,\{t\}, dv)\right)\{p^{-1}(1+\pi_t(\omega)^{\top}u)^p-p^{-1}\}.
\end{align}
Using \eqref{eq:tildeN1} together with~\eqref{eq:Phipi000}, we deduce that, for any $\pi\in\Gamma_{\cal C}^{Q^M}$,
		\begin{align}\label{eq:cond1122}
		{p}^{-1}\{V+(1+\Delta V)\cdot{N}({\pi})\}&=p^{-1}V_0+p^{-1}\frac{dV}{dA}\cdot A+[(1+\Delta V)\Phi_p^M(\pi)]\cdot A\nonumber\\
		&\quad+H^{\pi}\cdot R^{M,c} + W^{\pi}*(\mu-\nu^{M}).
		\end{align}
		It follows from \eqref{eq:1addDeltaA} that, for all $\pi\in\Gamma_{\cal C}^{Q^M}$, $H^{\pi}\in L_{loc}^{Q^M,2}(R^{M,c})$, $W^{\pi}\in G_{loc}^{Q^M}(\mu)$ and $|p^{-1}(1+\pi^{\top}u)^p-p^{-1}-\pi^{\top}h(u)|*\nu^{M}\in{\cal A}_{loc}^{Q^M,+}$. Moreover, it holds that, $Q^M(d\omega)\otimes dA_t$-a.s.
		\begin{align}\label{eq:optimcond}
		p^{-1}\frac{dV}{dA}+(1+\Delta V)\Phi_p^M({\pi}^*)=0,\qquad \sup_{\pi\in{\cal C}_0^{Q^M}\cap{\cal C}}\Phi_p^M(\pi)=\Phi_p^M({\pi}^*).
		\end{align}
Note that, by \eqref{eq:HWpi2}, we have $H^{\pi}=\pi$ and $W^{\pi}=p^{-1}(1+\pi_t(\omega)^{\top}u)^p-p^{-1}$ when $A_t=t$.	Then, by taking $A_t=t$, Eq.~\eqref{eq:optimcond000} follows from~\eqref{eq:optimcond},~\eqref{eq:1addDeltaA} and~\eqref{eq:hatAhatM}. Thus, we complete the proof of the lemma. \hfill$\Box$

\noindent{\bf Proof of Lemma~\ref{lem:recessionfcn}.}\quad
		(i) For $\pi\in\{y\in\R^n;\ F^{M}(\Lambda_{-}(y))=0\}$, note that $\{u\in\R_0^n;\ \pi^{\top}u+\lambda<0\}\uparrow \Lambda_{-}(\pi)$ as $\lambda\downarrow0$. This implies that $F^{M}(\{u\in\R_0^n;\ \pi^{\top}u+\lambda<0\})=0$ for all $\lambda>0$. Thus $\lambda^{-1}\pi\in C_0$ for all $\lambda>0$. In view of \eqref{eq:convexfcng}, we have
		\begin{align*}
		\lambda \psi(\lambda^{-1}\pi)&=-\pi^{\top}b^{M} -\frac{p-1}{2}\lambda^{-1}\pi^{\top}c\pi-\{p^{-1}\lambda^{1-p}(\lambda+\pi^{\top}u)^p-\lambda p^{-1}-\pi^{\top}h(u)\}*F^{M},
		\end{align*}
		for all $\lambda>0$. Then, from Corollary 8.5.2 in~\cite{Rockafellar}, it follows that $\psi0^+(\pi)=\lim_{\lambda\downarrow0}\lambda \psi(\lambda^{-1}\pi)$ and this gives \eqref{eq:g0+case1}. For $\pi\in\{y\in\R^n;\ F^{M}(\Lambda_{-}(y))>0\}$, because $\{u\in\R_0^n;\ \pi^{\top}u+\lambda<0\}\uparrow \Lambda_{-}(\pi)$ as $\lambda\downarrow0$, there exists a constant $\lambda_0>0$ such that for all positive $\lambda\leq\lambda_0$, it holds that $F^{M}(\{u\in\R_0^n;\ \pi^{\top}u+\lambda<0\})>0$. This implies that $\lambda^{-1}\pi\notin C_0$ for all positive $\lambda\leq\lambda_0$. By the definition of $\psi(\pi)$ in~\eqref{eq:convexfcng}, we have that $\lambda \psi(\lambda^{-1}\pi)=+\infty$ for all positive $\lambda\leq\lambda_0$. Hence $\psi0^+(\pi)=\lim_{\lambda\downarrow0}\lambda \psi(\lambda^{-1}\pi)=+\infty$. The proof of (ii) is similar to that of (i) and hence we omit it. (iii) It follows in a straightforward manner from previous arguments that the recession cone of $\psi$ is given by $\{\pi\in\R^n;\ \psi0^+(\pi)\leq0\}$. (iv) The direction of the vectors in the recession cone of $\psi$ is the direction of recession of $\psi$. The constancy space of $\psi$ is given by $\{\pi\in\R^n;\ \psi0^+(\pi)\leq0,\ \psi0^+(-\pi)\leq0\}$. It thus holds that $\{\pi\in\R^n;\ \psi0^+(\pi)\leq0,\ \psi0^+(-\pi)\leq0\}={\cal N}^{Q^M}$. \hfill$\Box$
\vspace{0.2cm}

\noindent{\bf Lemma~A.2}~{\bf(\cite{Nutzaap}, Appendix A)}\quad 
		{\it Let Assumption {\Af} hold. Then, the following statements hold:
		\begin{itemize}
			\item[{\rm(i)}] If $p\in(0,1)$, $I^M(\pi)$ defined by \eqref{eq:ipim} is finite and continuous on ${\cal C}_0$;
			\item[{\rm(ii)}] If $p<0$, $I_1^M(\pi)$ defined in \eqref{eq:ipim} is finite and continuous on ${\cal C}_0$, and $I_2^M(\pi)$ defined in \eqref{eq:ipim} is finite on $\cup_{\lambda\in[0,1)}\lambda{\cal C}_0$;
			\item[{\rm(iii)}] Under {\Ac}, any optimal strategy $\pi^*\in\argmax_{\pi\in{\cal C}\cap{\cal C}_0}\Phi_p^M(\pi)$ and is unique, modulo ${\cal N}^{Q^M}$.
		\end{itemize}}

\vspace{0.2cm}
\noindent{\bf Proof of Lemma~\ref{lem:HtildeHc}.}\quad  We provide the proof for a general increasing and predictable process {$A=(A_t)_{t\geq0}$ used in the representation of predictable characteristics for semimartingales}. Then, the desired result follows by taking $A_t=t$. It follows from \eqref{eq:M12} that
$$
\Delta M=\Xi(\Delta R)-\Xi(u)*\nu(\{\cdot\}, du)=e^{\theta(\Delta Y)}-1-\{e^{\theta(v)}-1\}*\nu^Y(\{\cdot\}, dv).
$$
Hence, the second identity follows from the assumption that $(R,Y)$ is QLC (if $A_t=t$, then it is automatically satisfied). We next prove the first identity. {Applying} \eqref{eq:M12} again, we obtain
\begin{align*}
		M=\frac{1}{2}(H\cdot R^c+\Lambda\cdot Y^c)+\frac{1}{2}\left(\Xi(u)*(\mu-\nu)+\{e^{\theta(v)}-1\}*(\mu^Y-\nu^Y)\right).
\end{align*}
Consequently, we obtain that
\begin{align}\label{eq:M1234}
\lc M^c,M^c\rc^P=\frac{1}{4}\lc H\cdot R^c+\Lambda\cdot Y^c,H\cdot R^c+\Lambda\cdot Y^c\rc^P=H^{\top}cH\cdot A=\Lambda^{\top}c^{Y}\Lambda\cdot A.
\end{align}
Moreover, it holds that
\begin{align*}
		\frac{1}{4}\lc H\cdot R^c+\Lambda\cdot Y^c,H\cdot R^c+\Lambda\cdot Y^c\rc^P&=\frac{1}{4}\left(H^{\top}cH+\Lambda^{\top}c^{Y}\Lambda+{H}^{\top}c^{RY}\Lambda+\Lambda^{\top}c^{YR}{H}\right)\cdot A\nonumber\\
		&=\frac{1}{2}\left(H^{\top}cH+{H}^{\top}c^{RY}\Lambda\right)\cdot A,
\end{align*}
where we use the fact that ${H}^{\top}c^{RY}\Lambda=\Lambda^{\top}c^{YR}{H}$. Then, \eqref{eq:M1234} implies that $\frac{1}{2}(H^{\top}cH+{H}^{\top}c^{RY}\Lambda)\cdot A=H^{\top}cH\cdot A$, and hence ${H}^{\top}c^{RY}\Lambda\cdot A=H^{\top}cH\cdot A$. The proof then follows by taking $A_t=t$ for $t\geq0$. \hfill$\Box$

In order to prove Theorem~\ref{thm:solhjb}, we need the following auxiliary result.

\noindent{\bf Lemma~A.3}\quad \emph{The following statements hold:
\begin{itemize}
\item[{\bf(i)}] Let $(\Pi^M,Z^M,W^M)$ satisfy the integral representation \eqref{eq:BSDEspecial} with initial value $\Pi_0^M=0$. If there exists a pair of functions $(\Pi,g)$ with $\Pi\in C^{1,2,1}$ and $g$ being a Borel function, such that $\Pi_t^M=\Pi(t,Y_t,\int_0^tg(s,Y_{s})ds)$, then $(\Pi,g)$ is a classical solution of the following forward equation: $\Pi(0,Y_0,0)=0$, and for $(t,y,z)\in[0,\infty)\times\R^d\times\R$,
		\begin{align}\label{eq:PDEPig}
		0&=\partial_t\Pi(t,y,z)+{\cal A}^Y\Pi(t,y,z)+g(t,y)\partial_z\Pi(t,y,z)
		+f(t,y,\nabla_y\Pi(t,y,z),\Pi(t,y+v,z)-\Pi(t,y,z)).
		\end{align}
		The function $f$ in \eqref{eq:BSDEspecial} is given by
		\begin{align}\label{eq:driver2222}
		f(t,y,Z,W) = \frac{1}{2}Z^{\top}c^{Y}(t,y)Z + p\varphi(t,y;Z,W)-\{W(v)+1-e^{W(v)}\}*F^Y(t,y),
		\end{align}
		where
		\begin{align}\label{eq:varphitytildeHh}
		&\varphi(t,y;Z,W):=\sup_{\pi\in C_t(y)\cap{\cal C}_{0,t}(y)}\Bigg\{\pi^{\top}\left(b(t,y) +c^{RY}(t,y)Z\right)+\frac{p-1}{2}\pi^{\top} c(t,y)\pi\\
		&\quad+\int\{p^{-1}(1+\pi^{\top}u)^p-p^{-1}-\pi^{\top}h(u)\}F(t,y,du)+\int\{p^{-1}(1+\pi^{\top}u)^p-p^{-1}\}\{e^{W(v)}-1\}\overline{F}(t,y,du,dv)\Bigg\}.\nonumber
		\end{align}
		The integral-differential operator ${\cal A}^Y$ is defined as:
		\begin{align}\label{eq:A}
		{\cal A}^Y\Pi(t,y,z)&:=\nabla_y\Pi(t,y,z)^{\top}b^Y(t,y)+\frac{1}{2}{\rm tr}[\nabla_{yy}^2\Pi(t,y,z) c^{Y}(t,y)]\nonumber\\
		&\quad+\{\Pi(t,y+v,z)-\Pi(t,y,z)-\nabla_y\Pi(t,y,z)^{\top}h^Y(v)\}*F^Y(t,y).
		\end{align}
\item[{\bf(ii)}] Assume the pair of functions $(\Pi,g)$, where $\Pi\in C^{1,2,1}$ and $g$ is a Borel function, satisfy the forward equation \eqref{eq:PDEPig} with initial condition $\Pi(0,Y_0,0)=0$. Then, $(\Pi^M,Z^M,W^M)$ with integral representation \eqref{eq:BSDEspecial} admits a factor representation with initial value $\Pi_0^M=0$.
	\end{itemize}}

\noindent{\bf Proof.}\quad (i) Suppose first that $\Pi_t^M=\Pi(t,Y_t,\int_0^tg(s,Y_{s})ds)$ for $t\geq0$, i.e., $\Pi^M$ has a factor representation. Let ${\cal Y}_t:=\int_0^tg(s,Y_{s})ds$ for $t\geq0$. Then, an application of It\^o formula yields that
		\begin{align}\label{eq:dPiM}
		d\Pi^M &= \left(\partial_t\Pi(Y_{-},{\cal Y}_{-})+{\cal A}^Y\Pi(Y_{-},{\cal Y}_{-})+ g(Y_{-})\partial_z\Pi(Y_{-},{\cal Y}_{-})\right)dt + \nabla_y\Pi(Y_{-},{\cal Y}_{-})dY^c\nonumber\\
		&\quad +\{\Pi(Y_{-}+v,{\cal Y}_{-})-\Pi(Y_{-},{\cal Y}_{-})\}*d(\mu^Y-\nu^Y).
		\end{align}
		Recall that $(\Pi^M,Z^M,W^M)$ satisfies the integral representation~\eqref{eq:BSDEspecial}, i.e., $d\Pi^M = -f(Y_{-},Z^M,W^M)dt + Z^M\cdot dY^c + W^M * d(\mu^Y-\nu^Y)$. Then, it holds that $Z^M=\nabla_y\Pi(Y_{-},{\cal Y}_{-})$, $W^M(v)=\Pi(Y_{-}+v,{\cal Y}_{-})-\Pi(Y_{-},{\cal Y}_{-})$, and
		\begin{align}\label{eq:fMN}
		-f(Y_{-},Z^M,W^M)=\partial_t\Pi(Y_{-},{\cal Y}_{-})+{\cal A}^Y\Pi(Y_{-},{\cal Y}_{-})+ g(Y_{-})\partial_z\Pi(Y_{-},{\cal Y}_{-}).
		\end{align}
		Recall that $\sup_{\pi\in{\cal C}\cap{\cal C}_0}\Phi^M(\pi)=\varphi(Y_{-})$, where the function $\varphi(t,y)$ is given by \eqref{eq:PhiMspecial}. It follows from Lemma~\ref{lem:HtildeHc} that  the function $\varphi(t,y)$ can be rewritten as:
		\begin{align*}
		&\varphi(t,y;\Lambda,\theta)=\sup_{\pi\in{C}_t(y)\cap{\cal C}_{0,t}(y)}\Bigg\{\pi^{\top}\left(b(t,y) +c^{RY}(t,y)\Lambda(t,y)\right)+\frac{p-1}{2}\pi^{\top} c(t,y)\pi\\ &\quad+\int\{p^{-1}(1+\pi^{\top}u)^p-p^{-1}-\pi^{\top}h(u)\}F(t,y,du)+\int\{p^{-1}(1+\pi^{\top}u)^p-p^{-1}\}\{e^{\theta(t,y,v)}-1\}\overline{F}(t,y,du,dv)\Bigg\}.\nonumber
		\end{align*}
		By Lemma~\ref{lem:PiMY}, the component $(Z^M,W^M)$ in \eqref{eq:BSDEspecial} is given by $Z^M=\Lambda$ and $W^M(v)=\theta(v)$.
		Hence, $Z^M=\Lambda=\nabla_y\Pi$ and  $W^M(v)=\theta(Y_{-},v)=\Pi(Y_{-}+v,{\cal Y}_{-})-\Pi(Y_{-},{\cal Y}_{-})$. Recall that the driver $f$ of \eqref{eq:BSDEspecial} is given by \eqref{eq:driverfspec}. Then,
		$f(Y_{-},Z^M,W^M) = \frac{1}{2}(Z^M)^{\top}c^{Y}Z^M + p\varphi(Y_{-},Z^M,W^M)-\{W^M(v)+1-e^{W^M(v)}\}*F^Y$.
		Plugging the above expression into~\eqref{eq:fMN}, we obtain the forward PDE~\eqref{eq:PDEPig}.
		
		(ii) Suppose that $(\Pi,g)$ is a classical solution of the forward equation \eqref{eq:PDEPig}. Then, Eq.~\eqref{eq:dPiM} reduces to
		\begin{align*}
		d\Pi(Y,{\cal Y}) &= -f(Y_{-},\nabla_y\Pi(Y_{-},{\cal Y}_{-}),\Pi(Y_{-}+v,{\cal Y}_{-})-\Pi(Y_{-},{\cal Y}_{-}))dt + \nabla_y\Pi(Y_{-},{\cal Y}_{-})dY^c\nonumber\\
		&\quad +\{\Pi(Y_{-}+v,{\cal Y}_{-})-\Pi(Y_{-},{\cal Y}_{-})\}*d(\mu^Y-\nu^Y).
		\end{align*}
		Using the above representation, a solution of \eqref{eq:BSDEspecial} is of the form: $\Pi^M=\Pi(Y,{\cal Y})$, $Z^M=\nabla_y\Pi(Y_{-},{\cal Y}_{-})$, and $W^M(v)=\Pi(Y_{-}+v,{\cal Y}_{-})-\Pi(Y_{-},{\cal Y}_{-})$. Hence, $(\Pi^M,Z^M,W^M)$ admits a factor representation. \hfill$\Box$

\noindent{\bf Proof of Theorem~\ref{thm:solhjb}.}\quad By applying the Cole-Hopf transform to~\eqref{eq:forhjbgamma} given by $\Gamma(t,y,z)=e^{L(t,y,z)}$ and we have that, for $(t,y,z)\in[0,\infty)\times\R^d\times\R$, 
	\begin{align}\label{eq:hjb3}
	0&=\partial_tL(t,y,z)+\partial_zL(t,y,z)g(t,y) + \nabla_yL(t,y,z)^{\top}b^Y(t,y) + \frac{1}{2}{\rm tr}[\nabla_{yy}^2L(t,y,z) c^{Y}(t,y)]\nonumber\\ &\quad+\frac{1}{2}\nabla_yL(t,y,z)^{\top}c^{Y}(t,y)\nabla_yL(t,y,z)
	+ p \sup_{\pi\in{C}_t(y)\cap{\cal C}_{0,t}(y)}\bigg\{\pi^{\top}\big(b(t,y)+c^{RY}(t,y)\nabla_yL(t,y,z)\big)+\frac{p-1}{2}\pi^{\top}c(t,y)\pi\nonumber\\
	&\quad+\{p^{-1}(1+\pi^{\top}u)^pe^{L(t,y+v,z)-L(t,y,z)}-p^{-1}-\pi^{\top}h(u)-p^{-1}\nabla_yL(t,y,z)^{\top}h^Y(v)\}*\overline{F}(t,y)\bigg\}.
	\end{align}
	We rewrite the terms in the last line of~\eqref{eq:hjb3} as:
	\begin{align*}
	&\{p^{-1}(1+\pi^{\top}u)^pe^{L(t,y+v,z)-L(t,y,z)}-p^{-1}-\pi^{\top}h(u)-p^{-1}\nabla_yL(t,y,z)^{\top}h^Y(v)\}*\overline{F}(t,y)\nonumber\\
	&\qquad=\{(p^{-1}(1+\pi^{\top}u)^p-p^{-1})(e^{L(t,y+v,z)-L(t,y,z)}-1)\}*\overline{F}(t,y)+\{p^{-1}(1+\pi^{\top}u)^p-p^{-1}-\pi^{\top}h(u)\}*F(t,y)\nonumber\\
	&\qquad\quad+p^{-1}\{e^{L(t,y+v,z)-L(t,y,z)}-1-\nabla_yL(t,y,z)^{\top}h^Y(v)\}*F^Y(t,y).
	\end{align*}
Then~\eqref{eq:hjb3} is reduced to
\begin{align}\label{eq:hjb4}
	0
&=\partial_tL(t,y,z)+\partial_zL(t,y,z)g(t,y) + {\cal A}^YL(t,y,z)+\frac{1}{2}\nabla_yL(t,y,z)^{\top}c^{Y}(t,y)\nabla_yL(t,y,z)\nonumber\\
&\quad+\{e^{L(t,y+v,z)-L(t,y,z)}-1-(L(t,y+v,z)-L(t,y,z))\}*F^Y(t,y)\nonumber\\
&\quad+ p \sup_{\pi\in{C}_t(y)\cap{\cal C}_{0,t}(y)}\bigg\{\pi^{\top}(b(t,y)+c^{RY}(t,y)\nabla_yL(t,y,z))+\frac{p-1}{2}\pi^{\top}c(t,y)\pi\\
&\quad+\{p^{-1}(1+\pi^{\top}u)^p-p^{-1}-\pi^{\top}h(u)\}*F(t,y)
	+\{(p^{-1}(1+\pi^{\top}u)^p-p^{-1})(e^{L(t,y+v,z)-L(t,y,z)}-1)\}*\overline{F}(t,y)\bigg\}.\nonumber
	\end{align}
	Recall the integral-differential operator ${\cal A}^Y $ defined by \eqref{eq:A}. Using the expression of the function $f$ given by \eqref{eq:driver2222} in Lemma~{A.3}, it follows  from~\eqref{eq:hjb4} that
	\begin{align*}
	0&=\partial_tL(t,y,z)+\partial_zL(t,y,z)g(t,y) + {\cal A}^YL(t,y,z)
	+f(t,y,\nabla_yL(t,y,z),L(t,y+v,z)-L(t,y,z)).
	\end{align*}
Using Lemma~A.3 again, $(\Pi,\tilde{g})$ satisfies the above equation. Then $(\Gamma,g)=(e^{\Pi},\tilde{g})$ is a solution of the forward HJB equation \eqref{eq:forhjbgamma}. This completes the proof of the theorem. \hfill$\Box$

\noindent{\bf Proof of Lemma~\ref{lem:PiMY}.}\quad For the local martingale $M$ given in \eqref{eq:M12}, Theorem~\ref{thm:existence} guarantees that $(\Pi^M,Z^M,W^M)$ with the integral representation~\eqref{eq:BSDEspecial} given by \eqref{eq:BSDE-sol} admits the form:
		\begin{align*}
		\Pi^M &=-p\int_0^{\cdot}\sup_{\pi\in{C}_s(y)\cap{\cal C}_{0,s}(y)}\Phi_p^M(\pi)ds+M-\frac{1}{2}\langle M^c,M^c\rangle^P-\{e^{\theta(Y_{-},v)}-1-\theta(Y_{-},v)\}*\mu^Y,\nonumber\\
		Z^M &=\Lambda(Y_{-}),\quad W^M(v)=\theta(Y_{-},v).
		\end{align*}
		We next compute $\int_0^{\cdot}\sup_{\pi\in{C}_s(y)\cap{\cal C}_{0,s}(y)}\Phi_p^M(\pi)ds$. 
Note that $\int_0^{\cdot}\sup_{\pi\in{C}_s(y)\cap{\cal C}_{0,s}(y)}\Phi_p^M(\pi)ds=\int_0^{\cdot}\tilde{\varphi}(Y_{s})ds$, where
		\begin{align}\label{eq:PhiMspecial222}
		\tilde{\varphi}(t,y)&=\sup_{\pi\in{C}_t(y)\cap{\cal C}_{0,t}(y)}\Bigg\{\pi^{\top}\left(b(t,y) +c(t,y)H(t,y)\right)+\frac{p-1}{2}\pi^{\top} c(t,y)\pi\\
		&\quad+\int\{p^{-1}(1+\pi^{\top}u)^p-p^{-1}-\pi^{\top}h(u)\}F(t,y,du)
		+\int\{p^{-1}(1+\pi^{\top}u)^p-p^{-1}\}\Xi(t,y,u)F(t,y,du)\Bigg\},\nonumber
		\end{align}
		{for $(t,y)\in[0,\infty)\times\R^d$}. Using the first identity in Lemma~\ref{lem:HtildeHc}, we obtain that
		\begin{align*}
		\{p^{-1}(1+\pi^{\top}u)^p-p^{-1}\}\Xi(u)*\mu&=\sum\{p^{-1}(1+\pi^{\top}\Delta R)^p-p^{-1}\}\Xi(\Delta R)\nonumber\\
		&=\sum\{p^{-1}(1+\pi^{\top}\Delta R)^p-p^{-1}\}\{e^{\theta(\Delta Y)}-1\}\nonumber\\
		&=\{p^{-1}(1+\pi^{\top}u)^p-p^{-1}\}\{e^{\theta(v)}-1\}*\overline{\mu}.
		\end{align*}
		By taking the dual predictable projection on both sides of the above equality, we  conclude that
		\begin{align*}
		\{p^{-1}(1+\pi^{\top}u)^p-p^{-1}\}\Xi(u)*F=\{p^{-1}(1+\pi^{\top}u)^p-p^{-1}\}\{e^{\theta(v)}-1\}*\overline{F}.
		\end{align*}
		By \eqref{eq:PhiMspecial222}, we deduce that $\tilde{\varphi}(t,y)={\varphi}(t,y)$. Then, the relation \eqref{eq:solutionspeci0} follows from \eqref{eq:M12}. \hfill$\Box$
	
\noindent{\bf Proof of Lemma~\ref{lem:PiMY22}.}\quad Since the factor $Y$ is a special semimartingale,  it admits the canonical representation
		\begin{align}\label{eq:Yrep}
		Y = Y_0 + B^Y + Y^c + v*(\mu^Y-\nu^Y).
		\end{align}
		If $\Lambda\equiv\sigma$ and $\theta(t,y,v)=\sigma^{\top}v$, it follows from \eqref{eq:Yrep} that
		\begin{align*}
		\Lambda\cdot Y^c+\theta(v)*(\mu^Y-\nu^Y)=\sigma^{\top}\{Y^c+v*(\mu^Y-\nu^Y)\}=\sigma^{\top}\left(Y-Y_0-\int_0^{\cdot}b_s^Yds\right).
		\end{align*}
		Plugging the equality above into~\eqref{eq:solutionspeci0}, we obtain \eqref{eq:solutionspeci022}. \hfill$\Box$\\

\section*{Acknowledgements}
We acknowledge the constructive suggestions of two anonymous referees, which contributed to improve the quality of this manuscript. {L. Bo has been  supported by NSFC grant no. 11971368 and National Center for Applied Mathematics of Shaanxi. A. Capponi has been supported in part by the NSF grant no. DMS-1716145. C. Zhou has been supported by Singapore MOE AcRF grant R-146-000-271-112 and NSFC grant 11871364.}

\end{APPENDICES}

\bibliographystyle{informs2014}
\bibliography{referencelist}

\end{document}